\documentclass[11pt,reqno]{amsart}
\usepackage{amsmath, amssymb, amsthm}
\usepackage{url}
\usepackage[breaklinks]{hyperref}

\newcommand{\df}{\dfrac}

 \renewcommand{\a}{\alpha}
\renewcommand{\b}{\beta}

\renewcommand{\d}{{\delta}}
\newcommand{\g}{\gamma}

\renewcommand{\l}{\lambda}

\renewcommand{\k}{\kappa}

\renewcommand{\(}{\left\(}
\renewcommand{\)}{\right\)}
\renewcommand{\[}{\left\[}
\renewcommand{\]}{\right\]}
\renewcommand{\i}{\infty}
\numberwithin{equation}{section}
 \theoremstyle{plain}
\newtheorem{theorem}{Theorem}[section]
\newtheorem{lemma}[theorem]{Lemma}
\newtheorem{remark}[]{Remark}

\newtheorem{conjecture}[theorem]{Conjecture}

\newtheorem{corollary}[theorem]{Corollary}

\newtheorem{example}[]{Example}

   \makeatletter
\def\proof{\@ifnextchar[{\@oproof}{\@nproof}}
\def\@oproof[#1][#2]{\trivlist\item[\hskip\labelsep\textit{#2 Proof of\
#1.}~]\ignorespaces}
\def\@nproof{\trivlist\item[\hskip\labelsep\textit{Proof.}~]\ignorespaces}

\makeatother

\begin{document}
\title[Untrodden pathways in the theory of the restricted partition function $p(n, N)$]{Untrodden pathways in the theory of the restricted partition function $p(n, N)$} 

\author{Atul Dixit, Pramod Eyyunni, Bibekananda Maji and Garima Sood}
\address{Discipline of Mathematics, Indian Institute of Technology Gandhinagar, Palaj, Gandhinagar 382355, Gujarat, India} 
\email{adixit@iitgn.ac.in}
\address{Harish-Chandra Research Institute, HBNI, Chhatnag road, Jhunsi, Allahabad, India, 211019} 
\email{pramodeyy@gmail.com}
\address{Discipline of Mathematics, Indian Institute of Technology Gandhinagar, Palaj, Gandhinagar 382355, Gujarat, India} 
\email{bibekananda.maji@iitgn.ac.in}
\address{Discipline of Mathematics, Indian Institute of Technology Gandhinagar, Palaj, Gandhinagar 382355, Gujarat, India} 
\email{garima.sood@iitgn.ac.in}
\begin{abstract}
We obtain a finite analogue of a recent generalization of an identity in Ramanujan's Notebooks. Differentiating it with respect to one of the parameters leads to a result whose limiting case gives a finite analogue of Andrews' famous identity for $\textup{spt}(n)$. The latter motivates us to extend the theory of the restricted partition function $p(n, N)$, namely, the number of partitions of $n$ with largest parts less than or equal to $N$, by obtaining the finite analogues of rank and crank for vector partitions as well as of the rank and crank moments. As an application of the identity for our finite analogue of the spt-function, namely $\textup{spt}(n, N)$, we prove an inequality between the finite second rank and crank moments. The other results obtained include finite analogues of a recent identity of Garvan, an identity relating $d(n, N)$ and lpt$(n, N)$, namely the finite analogues of the divisor and largest parts functions respectively, and a finite analogue of the Beck-Chern theorem. We also conjecture an inequality between the finite analogues of $k^{\textup{th}}$ rank and crank moments.
\end{abstract}
\maketitle
\tableofcontents
\section{Introduction}\label{intro}

The connection between basic hypergeometric series and generating functions of the divisor functions has been well explored. The earliest reference to it is probably of Kluyver \cite{kluyver} who proved that for $q\in\mathbb{C}, |q|<1$,
\begin{align}\label{Uchimura0}
\sum_{n=1}^{\infty} \frac{(-1)^{n-1} q^{\frac{n(n+1)}{2} } }{(1-q^n) ( q)_n  } = \sum_{n=1}^{\infty} \frac{ q^n }{1-q^n},
\end{align}
where, the notation followed here and throughout the sequel, is
\begin{align*}
&(A)_0:=(A;q)_0 =1;\quad (A)_n :=(A;q)_n  = (1-A)(1-Aq)\cdots(1-Aq^{n-1}),\quad n \geq 1,\\
&(A)_{\infty} :=(A;q)_{\i}  = \lim_{n\to\i}(A;q)_n\hspace{3mm} (|q|<1)\nonumber.
\end{align*}
The $(A;q)_n$ is called the $q$-shifted factorial and $q$ its base. When we simultaneously work with $q$-shifted factorials having different bases, we generally suppress the base $q$ in those $q$-shifted factorials having base $q$, but explicitly denote the bases other than $q$ in the other $q$-shifted factorials.

Equation \eqref{Uchimura0} was rediscovered by Fine \cite[p.~14, Equations (12.4), (12.42)]{fine}, and by Uchimura \cite[Theorem 2]{uchimura81} who also found an additional representation, namely,
\begin{align}\label{Uchimura}
\sum_{n=1}^{\infty} n q^n (q^{n+1} )_\infty = \sum_{n=1}^{\infty} \frac{(-1)^{n-1} q^{\frac{n(n+1)}{2} } }{(1-q^n) ( q)_n  } = \sum_{n=1}^{\infty} \frac{ q^n }{1-q^n}.
\end{align}
Identities such as these have been shown to have beautiful combinatorial interpretations arising from the theory of partitions. Before commencing on these, following are the notations used throughout the paper:

\begin{itemize}
\item $\pi$: an integer partition,

\item $|\pi|$: sum of the parts of $\pi$,

\item $p(n)$: the number of integer partitions of $n$,

\item $ s(\pi):=$ the smallest part of $\pi$,

\item $l(\pi):=$ the largest part of $\pi$,

\item $\#(\pi):=$ the number of parts of $\pi$,

\item $\mathrm{rank}(\pi)= l(\pi)- \#(\pi)$,
  
\item $\nu_{d}(\pi):=$ the number of parts of $\pi$ not counting multiplicity,

\item $\mathcal{P}(n):=$ collection of all integer partitions of $n$,

\item $\mathcal{D}(n):=$ collection of partitions of $n$ into distinct parts,

\item $p(n, N) := $ the number of integer partitions of $n$ such that $l(\pi) \leq N$,

\item $\mathcal{P}(n, N):=$ collection of all integer partitions of $n$ whose largest parts are $\leq N$. 
\end{itemize}
Bressoud and Subbarao \cite{bresub} showed that the equality of the first and the last expressions of \eqref{Uchimura} implies
\begin{align}\label{ffwidty}
\sum_{ \pi \in \mathcal{D}(n)  } (-1)^{ \# (\pi)-1 } s(\pi)=d(n),
\end{align}
where $d(n)$ denotes the number of divisors of $n$. Fokkink, Fokkink and Wang \cite{ffw95} also rediscovered this result by applying a beautiful combinatorial argument on a sequence of polynomials. Bressoud and Subbarao \cite{bresub} also generalized \eqref{ffwidty} for $\sigma_m(n)$, the sum of $m^{\textup{th}}$ powers of the divisors of $n$.

The finite analogues of identities such as \eqref{Uchimura} have also been well-studied. For example, van Hamme found that \cite{hamme} 
\begin{align}\label{hammeid}
\sum_{n=1}^{N} \frac{q^n}{1 - q^n} = \sum_{n=1}^{N}\left[\begin{matrix} N\\n\end{matrix}\right]\frac{(-1)^{n-1} q^{n(n+1)/2}}{(1-q^n)},
\end{align}
where
\begin{align*}
\left[\begin{matrix} N\\n\end{matrix}\right]=\left[\begin{matrix} N\\n\end{matrix}\right]_q :=\begin{cases}
\frac{(q;q)_N}{(q;q)_n (q;q)_{N-n}},\hspace{2mm}\text{if}\hspace{1mm}0\leq n\leq N,\\
0,\hspace{2mm}\text{otherwise},
\end{cases} 
\end{align*}
is the $q$-binomial coefficient. Letting $N\to\infty$ in \eqref{hammeid} easily gives the second equality in \eqref{Uchimura}. 
Guo and Zeng \cite[Equation (3.8)]{guozeng2015} obtained a finite analogue of the first and the last expressions in \eqref{Uchimura}, namely,
\begin{align}\label{uchigen}
\sum_{n=1}^{N} \frac{q^n}{1 - q^n} = \sum_{n=1}^\infty n q^n \left(  q^{n+1} \right)_{N-1} -\sum_{n=1}^\infty n q^{n+N} \left(  q^{n+1} \right)_{N-1},
\end{align}
and gave a refinement of \eqref{ffwidty} as follows.  Let $d(n, N)$ denote the number of divisors of $n$ which are less than or equal to $N$. Then
\begin{align}\label{guozeng}
 d(n, N) = t(n, N) - t(n-N, N),
 \end{align}
 where 
\begin{equation}\label{tnN}
t(n, N) := \sum_{\pi \in \mathcal{D}(n, N)} (-1)^{\#(\pi)-1} s(\pi),
\end{equation}
and $\mathcal{D}(n, N)$ is the collection of partitions of $n$ into distinct parts such that $l(\pi)-s(\pi) \leq N-1$. It is surprising that along with \eqref{guozeng}, Guo and Zeng do not give a combinatorial interpretation of the right-hand side of \eqref{hammeid}. However, as one might guess, \eqref{guozeng} itself is the combinatorial interpretation of \eqref{hammeid}. Indeed, the right-hand side of \eqref{hammeid} can be written in the form
\begin{align*}
&\sum_{n=1}^{N}\left[\begin{matrix} N\\n\end{matrix}\right]\frac{(-1)^{n-1} q^{n(n+1)/2}}{(1-q^n)}\nonumber\\
&=\sum_{n=1}^{N}\sum_{k=1}^{\infty}\left[\begin{matrix} N-1\\n-1\end{matrix}\right]k(-1)^{n-1}q^{nk+n(n-1)/2}-q^N\sum_{n=1}^{N}\sum_{k=1}^{\infty}\left[\begin{matrix} N-1\\n-1\end{matrix}\right]k(-1)^{n-1}q^{nk+n(n-1)/2}.
\end{align*}
Now from the fact \cite[p.~33]{gea1998} that $\left[\begin{matrix} N-1\\n-1\end{matrix}\right]$ is the generating function of the number of partitions of an integer into at most $n-1$ parts each less than or equal to $N-n$, one can see that 
\begin{equation*}
\left[\begin{matrix} N-1\\n-1\end{matrix}\right]q^{nk}\cdot q^{n(n-1)/2}
\end{equation*}
is the generating function for partitions into $n$ distinct parts with smallest part $k$ such that 
\begin{equation*}
l(\pi)-s(\pi)\leq (N-n)+n-1=N-1.
\end{equation*}
Thus the right-hand side of \eqref{hammeid} is the generating function of $t(n, N)-t(n-N,N)$. This establishes \eqref{guozeng} since $\sum_{n=1}^{N} \frac{q^n}{1 - q^n}$ generates $d(n, N)$.

It must be mentioned here that along with the finite analogues \eqref{hammeid} and \eqref{uchigen}, there also exists a one-variable generalization of \eqref{Uchimura} in the literature. In fact, it is an identity in Ramanujan's Notebook \cite[p.~264, Entry 4]{bcbramforthnote}, \cite[p.~354]{ramanujanoriginalnotebook2}, \cite[p.~302-303]{ramanujantifr}, namely,
\begin{equation}\label{Garvan's identity}
\sum_{n=1}^{\infty} \frac{(-1)^{n-1} z^n q^{\frac{n(n+1)}{2} } }{(1-q^n) (z q)_n  } = \sum_{n=1}^{\infty} \frac{z^n q^n }{1-q^n},
\end{equation}
where $z\neq q^{-n}, n\geq 1$. It was rediscovered by Uchimura \cite[Equation (3)]{uchimura81} and Garvan \cite{garvan0}. 

Identity \eqref{Garvan's identity} was recently generalized further by the first and the third authors in \cite[Theorem 2.2]{dixitmaji18} by obtaining for $|zq|<1$ and $|c|\leq 1$,
\begin{align}\label{gen of Garavan}
\sum_{n=1}^{\infty} \frac{(-1)^{n-1} z^n q^{\frac{n(n+1)}{2} } }{(1-c q^n) (z q)_n  } = \frac{z}{c}\sum_{n=1}^{\infty}\frac{(zq/c)_{n-1}}{(zq)_n}(cq)^n.
\end{align}
The case $z=1$ of this identity, with the right-hand side expressed as a $q$-product by $q$-binomial theorem, was previously obtained by Yan and Fu \cite{yanfunanjing} and was rediscovered by Andrews, Garvan and Liang \cite[Theorem 3.5]{agl13} by generalizing the left-hand side of \eqref{ffwidty} to
\begin{equation*}
\textup{FFW}(c,n):=\sum_{\pi\in\mathcal{D}(n)}(-1)^{\#(\pi)-1}\left(1+c+\cdots+c^{s(\pi)-1}\right).
\end{equation*}
For an up-to-date history of these and other such identities, we refer the reader to \cite{dixitmaji18}. 

Ismail and Stanton \cite{ismailstantonanncomb} observed that the genesis of such identities is in the theory of basic hypergeometric functions. Indeed, as an application of a ${}_3\phi_{2}$-transformation \cite[p.~359, (III.9)]{gasperrahman}, the result
\begin{align}\label{entry3gen}
\sum_{n =1}^{\infty}  \frac{ (b/a)_n a^n }{ (1- c q^n) (b)_n } =  \sum_{m=0}^{\infty}\frac{(b/c)_mc^m}{(b)_m}\left(\frac{aq^m}{1-aq^m}-\frac{bq^m}{1-bq^m}\right)
\end{align}
was obtained in \cite[Theorem 2.1]{dixitmaji18} for $|a|<1$, $|b|<1$ and $|c|\leq 1$. Then \eqref{gen of Garavan} was derived from it as a special case by letting $a\to 0$ and replacing $b$ by $zq$. The richness of partition-theoretic information embedded in \eqref{gen of Garavan} and other related identities are demonstrated in the same paper.

One of the goals of this paper is to obtain a finite analogue of \eqref{entry3gen}. As we shall see, this finite analogue gives many important corollaries, one of which is a source for most of the content in the sequel. We begin with the finite analogue of \eqref{entry3gen}.
\begin{theorem}\label{finmainabc}
Let $N\in\mathbb{N}$. For $a, b, c\neq q^{-n}, 1\leq n\leq N-1$, and $c\neq q^{-N}$,
\begin{align}\label{finmainabceqn}
\sum_{n=1}^{N}\left[\begin{matrix} N\\n\end{matrix}\right]\frac{(\frac{b}{a})_{n}(q)_{n}(a)_{N-n}a^{n}}{(1-cq^{n})(b)_n(a)_N}=
\sum_{n=1}^{N}\left[\begin{matrix} N\\n\end{matrix}\right]\frac{(\frac{b}{c})_{n-1}(q)_n (cq)_{N-n}c^{n-1}}{(b)_{n-1}(cq)_N}\left(\frac{aq^{n-1}}{1-aq^{n-1}}-\frac{bq^{n-1}}{1-bq^{n-1}}\right).
\end{align}
\end{theorem}
Letting $a\rightarrow 0$, replacing $b$ by $zq$ in Theorem \ref{finmainabc}, we obtain a finite analogue of \eqref{gen of Garavan}. We record it separately as a theorem as it will be frequently used in the sequel.
\begin{theorem}\label{fingGaravan}
Let $N\in\mathbb{N}$. For $z, c\neq q^{-n}, 1\leq n\leq N$,
\begin{equation}\label{fingen of Garavan}
\sum_{n=1}^{N}\left[\begin{matrix} N\\n\end{matrix}\right]\frac{(-1)^{n-1}z^nq^{\frac{n(n+1)}{2}}(q)_{n}}{(1-cq^{n})(zq)_n}=
\frac{z}{c}\sum_{n=1}^{N}\left[\begin{matrix} N\\n\end{matrix}\right]\frac{(zq/c)_{n-1}(q)_{n} (cq)_{N-n} (cq)^{n}}{(zq)_{n}(cq)_{N}}.
\end{equation}
\end{theorem}
Another corollary of Theorem \ref{finmainabc}, which generalizes an identity of Ramanujan \cite[p.~355]{ramanujanoriginalnotebook2}, \cite[p.~265, Entry 5]{bcbramforthnote}, \cite[p.~302-303]{ramanujantifr} is discussed in Section \ref{prooffinmainabc}.

The interesting special cases of Theorem \ref{fingGaravan}, which include finite analogues of Ramanujan's identity \eqref{Garvan's identity}, Yan and Fu's identity \cite[p.~116, Equation (4)]{yanfunanjing} as well as that involving a generalization of the finite analogue of Ramanujan's celebrated function 
\begin{equation}\label{sigq}
\sigma(q):=\sum_{n=0}^{\infty} \frac{q^{n(n+1)/2}}{(-q)_n},
\end{equation}
are discussed in Section \ref{prooffinmainabc}.

In a recent paper, Garvan \cite[Equation (1.3)]{garvan1} derived an interesting identity and gave combinatorial implications of its corollaries. For $|z|\leq 1$ and $|q|<1$, this identity is
\begin{align}\label{new identity}
\sum_{n=1}^{\infty}  \frac{(-1)^{n-1} z^n q^{n^2} }{(zq;q^2)_{n} (1-z q^{2n}) } 
=\sum_{n=1}^{\infty} \frac{z^n q^{ \frac{n(n+1)}{2}} (q;q)_{n-1} }{ (zq; q)_n  }.
\end{align}
A  natural proof of Garvan was obtained in \cite[Equations (6.1), (6.4), (6.6)]{dixitmaji18} using \eqref{gen of Garavan}. Theorem \ref{fingGaravan} can be used to obtain a finite analogue of Garvan's identity, namely, 
\begin{theorem}\label{fingithm}
Let $N\in\mathbb{N}$. For $z\neq q^{-n}, 1\leq n\leq 4N-1$,
\begin{align}
&\sum_{n=1}^{N}\left[\begin{matrix} N\\n\end{matrix}\right]_{q^2}\frac{(-1)^{n-1}z^nq^{n^2}(q^2;q^2)_{n}}{(zq;q^2)_n(1-zq^{2n})}\nonumber\\
&=\sum_{n=1}^{N}\left[\begin{matrix} N\\n\end{matrix}\right]_{q^2}\left(\frac{(q;q)_{2n-2}z^{2n-1}q^{n(2n-1)}}{(zq;q)_{2n-1}}+\frac{(q;q)_{2n-1}z^{2n}q^{n(2n+1)}}{(zq;q)_{2n}}\right)\frac{(q^{2};q^2)_n}{(zq^{2N+1};q^2)_{n}}.\label{ga3s}
\end{align}
\end{theorem}
\begin{remark}
Letting $N\to\infty$ in the above theorem gives \eqref{new identity}, for, $\displaystyle\lim_{N\to\infty}\left[\begin{matrix} N\\n\end{matrix}\right]_{q^2}\frac{(q^{2};q^2)_n}{(zq^{2N+1};q^2)_{n}}=1$ and then the two expressions inside the parenthesis on the right-hand side of \eqref{ga3s} beautifully combine resulting in $\displaystyle\sum_{n=1}^{\infty} \frac{z^n q^{ \frac{n(n+1)}{2}} (q;q)_{n-1} }{ (zq; q)_n  }$.
\end{remark}
We now state an important result which, as alluded to above, is the genesis of the most of the content of this paper. This result is obtained by taking the first derivative of \eqref{fingen of Garavan} with respect to $z$ and then letting $z\to 1$. It is a finite analogue of Theorem 2.8 from \cite{dixitmaji18}.
\begin{theorem}\label{fingenc}
Let Fine's function $F(a,b;t)$ be defined by \cite[p.~1]{fine}
\begin{equation}\label{finefunction}
F(a,b;t):=\sum_{n=0}^{\infty}\frac{(a q)_n}{(b q)_n}t^n.
\end{equation}
Let $N\in\mathbb{N}$. Then for $|q|<1, |c|<1/|q|, c\neq q^{-n}, 0\leq n\leq N$,
\begin{align}\label{finanalogmaineqn}
 & \frac{1}{(q)_N}\sum_{n = 1}^{N} \left[\begin{matrix} N \\ n \end{matrix}\right]\frac{(-1)^{n - 1}nq^{n(n + 1)/2}}{1 - cq^n} + 
\sum_{n = 1}^{N}\left[\begin{matrix} N \nonumber \\ n \end{matrix}\right] \frac{q^{n(n + 1)}}{(q)_n (1 - q^n)} F\left(q^N, q^n; cq^n\right) \\ 
& = \ \frac{c}{(1 - c)^2 (q)_N}\left(\frac{(q)_N}{(cq)_N} - 1\right) + \frac{1}{(c)_{N+1}}\sum_{n = 1}^{N}\frac{(cq)_n}{(q)_n}\frac{q^n}{1 - q^n}. 
\end{align}
\end{theorem}

\section{New results in the theory of the restricted partition function $p(n, N)$}\label{new}

Let $p(n, N)$ denote the number of partitions of $n$ whose largest parts are less than or equal to $N$. Although not as popular as the partition function $p(n)$ itself, the restricted partition function $p(n, N)$ has been studied by many mathematicians. Szekeres \cite{szekeres1}, \cite{szekeres2} proved an asymptotic formula for $p(n, N)$ whereas Almkvist and Andrews \cite{almand} obtained a Hardy-Ramanujan-Rademacher-type formula for it. Kronholm and Rehmert \cite[Theorem 1]{kronholmrehmert} obtained a general congruence for $p(n, N)$, namely, if $N$ is an odd prime, $k\geq 1, 1\leq j\leq\frac{N-1}{2}$, $\a\geq 1$, and if $\textup{lcm}(a)$ denotes the least common multiple of the natural numbers from $1$ to $a$, then
\begin{equation}\label{kronreh}
p(\textup{lcm}(N)N^{\a-1}k-jN,N)\equiv0\pmod{N^{\a}}.
\end{equation}
However, to the best of our knowledge, there isn't any literature on the partition statistics for $p(n, N)$ similar to that for $p(n)$.

The second goal of this paper is to extend the theory of $p(n, N)$ to include not only the corresponding rank and crank in terms of the vector partitions associated with $p(n, N)$, and the rank and crank moments, but also the associated smallest parts function which we denote by $\textup{spt}(n, N)$. 

\noindent
\underline{\textbf{Definition 1}} $\textup{spt}(n, N):=$ the number of smallest parts in all partitions of $n$ whose corresponding largest parts are less than or equal to $N$.

 Clearly, $\textup{spt}(n, 1)=n$, and for $n\leq N$, $\textup{spt}(n, N)=\textup{spt}(n)$.

The motivation for this extension of the theory of $p(n, N)$ lies in the fact that the following special case when $c\to 1$ of Theorem \ref{fingenc} is actually, as we shall see, the generating function version of the finite analogue of Andrews' identity for $\textup{spt}$-function (see Theorem \ref{fin_spt_identity} below).  
\begin{theorem}\label{fin_spt_q_identity}
Let $N\in\mathbb{N}$. Then for $|q|<1$,
\begin{align*}
 \frac{1}{(q)_N}\sum_{n = 1}^{N} \left[\begin{matrix} N \\ n \end{matrix}\right]\frac{(-1)^{n - 1}nq^{n(n + 1)/2}}{1 - q^n} & + \sum_{n = 1}^{N} \left[\begin{matrix} N \\ n \end{matrix}\right] \frac{q^{n^2}}{(q)_n} \sum_{k=1}^n \frac{q^k}{(1-q^k)^2 } = \frac{1}{(q)_N} \sum_{n=1}^{\infty} \frac{n q^n(1- q^{N n})}{1- q^n}. 
\end{align*}
\end{theorem}
However, before we proceed with the finite analogue of Andrews' identity, it makes sense to first introduce the finite analogues of rank, crank and their moments. The introduction of these new concepts when compared with the ones for $p(n)$, viewed historically, is like moving in the reverse direction, for, the rank and crank were introduced only after partition congruences modulo $5, 7$ and $11$ were found by Ramanujan with a need to explain as to why they exist. Nevertheless, as we shall see, the introduction of the finite analogues of rank and crank will be very fruitful in the development of the theory of $p(n, N)$.

Let $V_{1}=\mathcal{D}\times\mathcal{P}$ denote the set of vector partitions, where $\mathcal{D}$ is the set of partitions of a number into distinct parts and $\mathcal{P}$ is the set of unrestricted partitions. Denote an element $\vec{\pi}$ of $V_1$ by $(\pi_1, \pi_2)$, where $|\vec{\pi}|=|\pi_1|+|\pi_2|$. For any positive integer $N$ and $j$ with $1\leq j \leq N$, define $S_1$ to be
\begin{align}\label{S1}
S_1&:=\bigg\{\vec{\pi}\in V_1: \pi_1\hspace{1mm}\text{is either an empty partition or such that its parts lie in}\hspace{1mm}[N-j+1, N]\hspace{1mm}\nonumber\\
&\quad\quad\text{and}\hspace{1mm}\pi_2\hspace{1mm}\text{is an unrestricted partition with its Durfee square of size}\hspace{1mm} j\bigg\}.
\end{align}
Let $w_r(\vec{\pi}) = (-1)^{\#(\pi_1)} $ be the weight of the vector partition $\vec{\pi}$ and its $\textup{rank}(\vec{\pi})=\textup{rank}(\pi_2)$. Now define 
\begin{equation}\label{ns1}
N_{S_1}(m, n):=\sum_{j=1}^N N_{S_1}\left(m, n; \boxed{j}\right),
\end{equation}
where
\begin{equation*}
N_{S_1}\left(m, n; \boxed{j}\right):= \sum_{\vec{\pi} \in S_1, |\vec{\pi}|=n \atop \mathrm{rank}(\vec{\pi})=m} w_r(\vec{\pi}),   
\end{equation*}
that is, $N_{S_1}\left(m, n; \boxed{j}\right)$ is the number of vector partitions of $n$ with rank $m$ and counted with weight $w_r(\vec{\pi})$, and with the size of the Durfee squares of $\pi_2$ equal to $j$. 

One can easily verify that for any fixed $j$, $N_{S_1}\left(m, n; \boxed{j}\right) =N_{S_1}\left(-m, n; \boxed{j}\right)$ and hence $N_{S_1}(m, n) =N_{S_1}(-m, n)$. For a fixed $j$, $1\leq j\leq N$, consider 
$\left[\begin{matrix} N \\ j \end{matrix}\right] \frac{q^{j^2} (q)_j }{(z q)_j (z^{-1} q)_j}$.
Since $\left[\begin{matrix} N \\ j \end{matrix}\right](q)_j=(q^{N-j+1})_j$, it is easy to see that it generates partitions $\pi_1$ described in \eqref{S1}. Also, $\frac{q^{j^2} }{(z q)_j (z^{-1} q)_j}$ generates partitions $\pi_2$ with power of $z$ keeping track of its rank. Hence we have the following generating function for $N_{S_1}(m, n)$.
\begin{theorem}\label{frgfthm}
Let $N\in\mathbb{N}$. Then
\begin{equation}\label{fin_rank_gen}
\sum_{j=1}^N \left[\begin{matrix} N \\ j \end{matrix}\right] \frac{q^{j^2} (q)_j }{(z q)_j (z^{-1} q)_j}=\sum_{n=1}^{\infty}\sum_{m=-\infty}^{\infty} N_{S_1}(m, n) z^m q^n.
\end{equation}
\end{theorem}
We call the left-hand side of \eqref{fin_rank_gen} the finite analogue of the rank generating function, for letting $N \rightarrow \infty$ on  both sides of \eqref{fin_rank_gen}, gives the well-known result for the rank generating function, namely, if $N(m,n)$ denote the number of partitions of $n$ with rank $m$, then \cite[p.~66]{garvan88}
\begin{equation*}
\sum_{n=1}^{\infty}  \sum_{m=-\infty}^{\infty} N\left(m, n\right) z^m q^n =  \sum_{j=1}^\infty  \frac{q^{j^2}  }{(z q)_j (z^{-1} q)_j}.
\end{equation*}
We also observe that when $n\leq N$, $N_{S_1}(m, n)=N(m, n)$. Indeed, $0\leq|\pi_1|\leq n-j^2$ since $\pi_2$ contains at least a Durfee square of size $j$. However, $n-j^2\leq N-j<N-j+1$ implies that $\pi_1$ must be an empty partition and hence if $Q(m, n, j)$ denote the number of partitions of $n$ with rank $m$ and size of their Durfee squares $j$, then
\begin{equation*}
N_{S_1}(m, n)=\sum_{j=1}^N N_{S_1}\left(m, n; \boxed{j}\right)=\sum_{j=1}^{\left\lfloor\sqrt{n}\right\rfloor}Q(m, n, j)=N(m, n).
\end{equation*}
Moreover, note that using \cite[Equation (16)]{bressoud} (see also \cite[Section 265, p.~26]{macmahon}), when $z=1$, the left-hand side of \eqref{fin_rank_gen} reduces to $1/(q)_N$, the generating function of $p(n, N)$, which is analogous to the rank-generating function reducing to $1/(q)_{\infty}$, the generating function of $p(n)$.

While the rank of a partition could explain two of Ramanujan's three partition congruences, namely, $p(5n+4)\equiv 0\hspace{1mm}(\textup{mod}\hspace{1mm} 5)$ and $p(7n+5)\equiv 0\hspace{1mm}(\textup{mod}\hspace{1mm} 7)$, it could not explain his third congruence $p(11n+6)\equiv 0\hspace{1mm}(\textup{mod}\hspace{1mm} 11)$. Hence Dyson \cite{dys} hypothesized the existence of another partition statistic that he called `crank' that would do the job. Garvan \cite{garvan88} found the crank for vector partitions, and finally it was Andrews and Garvan \cite{andrewsgarvan88} together who found the crank for an ordinary partition. They proved that
\begin{equation*}
\frac{(q)_{\infty}}{(zq)_\infty(z^{-1}q)_\infty} = \sum_{n=0}^{\infty} \sum_{m=-\infty}^{\infty} M(m, n) z^m q^n,
\end{equation*}
where $M(m,n)$ is the number of partitions of $n$ with crank $m$. In what follows, we obtain a finite analogue of the crank for vector partitions. 

Let $V_{2}=\mathcal{D}\times\mathcal{P}\times\mathcal{P}$ denote the set of vector partitions, where $\mathcal{D}$ denotes the set of partitions into distinct parts and $\mathcal{P}$ denotes the set of partitions. Denote an element $\vec{\pi}$ of $V_2$ by $(\pi_1, \pi_2, \pi_3)$ such that $|\vec{\pi}|=|\pi_1|+|\pi_2|+|\pi_3|$. For any positive integer $N$, define $S_2$ by
\begin{equation*}
S_2:=\{\vec{\pi}\in V_2 : l(\pi_1), l(\pi_2), l(\pi_3) \leq N \}. 
\end{equation*}
Define $w_c(\vec{\pi})=(-1)^{\#(\pi_1)}$ to be the weight of the vector partition $\vec{\pi}$ and crank$(\vec{\pi})= \#(\pi_2)-\#(\pi_3)$. Define
\begin{equation}\label{ms1}
M_{S_2}(m, n) := \sum_{\vec{\pi} \in S_2, |\vec{\pi}|=n \atop \mathrm{crank}(\vec{\pi})=m} w_c(\vec{\pi}),   
\end{equation}
that is, $M_{S_2}(m, n)$ denotes the number of vector partitions of $n$ with crank $m$ counted with weight $w_c(\vec{\pi})$. This implies $M_{S_2}(m, n)=M_{S_2}(-m, n)$. In view of the above, we have the following generating function for $M_{S_2}(m, n)$.
\begin{theorem}\label{fcgfthm} Let $N\in\mathbb{N}$. Then
\begin{equation}\label{fcgfeqn}
\frac{(q)_N}{(z q)_N(z^{-1}q)_N} = \sum_{n=0}^\infty\sum_{m=-\infty}^{\infty} M_{S_2}(m, n) z^m q^n. 
\end{equation}
\end{theorem}
The expressions on the left-hand sides of \eqref{fin_rank_gen} and \eqref{fcgfeqn} have appeared in the literature, however, as far as we know, they have not been studied from a combinatorial standpoint. Indeed, from \cite[p.~252, Theorem 2.1]{andpar}, \cite[p.~263]{abramlostI}, 
\begin{align}\label{frgfbil}
\sum_{n=0}^{N}\left[\begin{matrix} N \\ n \end{matrix}\right]\frac{(q)_nq^{n^2}}{(zq)_n(z^{-1}q)_n}=\frac{1}{(q)_N}+(1-z)\sum_{n=1}^{N}\left[\begin{matrix} N \\ n \end{matrix}\right]\frac{(-1)^n(q)_nq^{n(3n+1)/2}}{(q)_{n+N}}\left(\frac{1}{1-zq^n}-\frac{1}{z-q^n}\right),
\end{align}
and Andrews \cite[p.~258, Theorem 4.1]{andpar} showed that
\begin{align}\label{fcgfbil}
\frac{(q)_{N}}{(zq)_{N}(z^{-1}q)_{N}}=\frac{1}{(q)_N}+(1-z)\sum_{n=1}^{N}\left[\begin{matrix} N \\ n \end{matrix}\right]\frac{(-1)^n(q)_nq^{n(n+1)/2}}{(q)_{n+N}}\left(\frac{1}{1-zq^n}-\frac{1}{z-q^n}\right).
\end{align}
Note that the right-hand sides of \eqref{frgfbil} and \eqref{fcgfbil} are respectively the finite analogues of the bilateral series representations for the rank and crank generating functions \cite[Equations (12.2.3), (12.2.9)]{abramlostI}.

For any positive integer $k$, we define the finite analogues of $k^{\textup{th}}$ rank and crank moments respectively by
\begin{align}
N_{k, N}(n) &:= \sum_{m=-\infty}^{\infty} m^k N_{S_1}(m, n),\label{finrankmom}\\
M_{k,N}(n)&:= \sum_{m=-\infty}^{\infty} m^k M_{S_2}(m, n),\label{fincrankmom}
\end{align}
where $N_{S_1}(m, n)$ and $M_{S_2}(m, n)$ are defined in \eqref{ns1} and \eqref{ms1} respectively. It is easy to see that the odd finite rank and crank moments are equal to zero. 

In 2008, Andrews \cite{andrews08} introduced smallest part partition function $\textup{spt}(n)$ to be the total number of appearances of the smallest parts in all partitions of $n$ and proved that
\begin{equation}\label{Andrews_spt0}
\mathrm{spt}(n)= n p(n) - \frac{1}{2} N_{2}(n).
\end{equation}
In view of Dyson's identity \cite[Theorem 5]{dyson89}
\begin{equation}\label{dysid}
n p(n) = \frac{1}{2} M_{2}(n),
\end{equation}
this implies that
\begin{equation}\label{Andrews_spt}
\mathrm{spt}(n)=  \frac{1}{2} \left( M_{2}(n) -  N_{2}(n) \right).
\end{equation}
We note in passing that Dyson's identity \eqref{dysid} was implicitly derived by Andrews \cite[p.~136]{andrews08} in its analytical form.

We are now ready to state the finite analogue of \eqref{Andrews_spt}.
\begin{theorem}\label{fin_spt_identity}
For any natural number $N$, we have
\begin{align}\label{1}
\mathrm{spt}(n, N) = \frac{1}{2}\left(M_{2,N}(n)- N_{2,N}(n)   \right).
\end{align}
In other words,
\begin{align}\label{2}
\mathrm{spt}(n, N) = \sum_{j=0}^{n-1}p(j, N)\sigma(n-j, N)-\frac{1}{2}N_{2,N}(n),
\end{align}
where
\begin{equation}\label{skn}
\sigma(k, N)=\sum_{d|k\atop d\geq k/N}d.
\end{equation}
\end{theorem}
Letting $N \rightarrow \infty$ in either \eqref{1} or \eqref{2}, we obtain Andrew's spt-identity \eqref{Andrews_spt}. To obtain \eqref{Andrews_spt} from \eqref{2}, one additionally needs to use Euler's recurrence relation $np(n)=\sum_{j=0}^{n-1}p(j)\sigma(n-j)$. We now illustrate \eqref{2}.
\begin{example}
Let $N=3$ and $n=6$. One can check that $\textup{spt}(6,3)=21$ and $ \sum_{j=0}^{5} p(j, 3) \sigma(6-j, 3)=1(11)+1(5)+2(6)+3(4)+4(3)+5(1)+7(0)=57$. Next,
\begin{align*}
N_{2, 3}(n)&=\sum_{m=-5}^{5}m^2\sum_{j=1}^{3} N_{S_1}(m, 6; \boxed{j})\nonumber\\
&=50\sum_{j=1}^{3} N_{S_1}(5, 6; \boxed{j})+32\sum_{j=1}^{3} N_{S_1}(4, 6; \boxed{j})+18\sum_{j=1}^{3} N_{S_1}(3, 6; \boxed{j})\nonumber\\
&\quad+8\sum_{j=1}^{3} N_{S_1}(2, 6; \boxed{j})+2\sum_{j=1}^{3} N_{S_1}(1, 6; \boxed{j})\nonumber\\
&=50(1)+32(0)+18(1)+8(0)+2(2)\nonumber\\
&=72.
\end{align*}
Thus,
\begin{equation*}
\sum_{j=0}^{5} p(j, 3) \sigma(6-j, 3)-\frac{1}{2}N_{2, 3}(n)=57-36=21=\textup{spt}(6, 3),
\end{equation*}
as stated by \eqref{2}.
\end{example}
An immediate application of Theorem \ref{fin_spt_identity} is an inequality between the finite analogues of rank and crank moments.
\begin{corollary}\label{inq_fin_crank_rank}
Let $N\in\mathbb{N}$ be fixed. Then for all $n\geq 1$,
\begin{equation*}
M_{2,N}(n) > N_{2,N}(n).  
\end{equation*}
\end{corollary}
The asymptotic estimate of $\textup{spt}(n, N)$ is now given. 
\begin{theorem}\label{asyest}
For any positive integer $N$, as $n \rightarrow \infty$,
\begin{align*}
\textup{spt}(n, N) = \frac{n^N}{(N!)^2} + O(n^{N-1}). 
\end{align*}
\end{theorem}
As will be shown in this paper, one could also follow Andrews' approach in \cite{andrews08} to derive Theorem \ref{fin_spt_identity}, however, deriving it from Theorem \ref{fingenc} gives a uniform approach in obtaining such identities. Indeed, Theorem \ref{fin_spt_q_identity} is but one special case, namely when $c\to 1$, of Theorem \ref{fingenc}. Its other special case $c=-1$ is discussed in Section \ref{nscsection}. It is concerned with a finite analogue of $N_{\textup{SC}}(n)$, the number of self-conjugate $S$-partitions counted with a certain specific weight \cite[Equation (1.10)]{agl13}. Similarly the case $c=0$, discussed in Section \ref{furcor}, gives a relation between $d(n, N)$ and $\textup{lpt}(n, N)$, the number of occurrences of the largest parts in those partitions $\pi$ of $n$ in which the largest parts are $\leq N$.

In order to derive our results involving the finite analogue of $N_{\textup{SC}}(n)$, it became imperative to generalize a result of George Beck \cite[A034296]{beck} recently proved by Chern \cite[Theorem 1.2]{chern}. Let $\textup{ssptd}_{o}(n)$ denote the sum of the smallest parts in all partitions of n into distinct parts which are odd in number and let $a(n)$ denote the number of compact partitions of $n$ \cite{andrewsba}, that is, the partitions in which every number between their largest and smallest parts also appears as a part, or by conjugation, the number of partitions of $n$ in which only the largest part can repeat. Then the Beck-Chern result is that $a(n)=\textup{ssptd}_{o}(n)$. Our finite analogue of this theorem is now given.
\begin{theorem}\label{chernbeck}
Let $N\in\mathbb{N}$. Let $a(n, N)$ denote the number of partitions of $n$ with $l(\pi) \leq N$ in which only the largest part may repeat. Let $\textup{ssptd}_{o}(n, N)$ denote the sum of smallest parts in all partitions $\pi$ of $n$ into odd number of distinct parts, and satisfying $l(\pi) - s(\pi)\leq N-1$. Then
\begin{equation}\label{bcgen}
a(n, N) = \textup{ssptd}_{o}(n, N) - \textup{ssptd}_{o}(n-N, N).
\end{equation}
\end{theorem}
The form of the above finite analogue is reminiscent of \eqref{guozeng}. Of course, the Beck-Chern result is immediately recovered upon letting $N\to\infty$ in \eqref{bcgen}, or when $n\leq N$.
\begin{example}
Here is an example illustrating Theorem \ref{chernbeck}. Let $N=3$ and $n=8$. The relevant partitions enumerated by $a(8,3)$ are $2+2+2+2,~3+3+2,~1+1+1+1+1+1+1+1$. Thus $a(8,3)=3$. Now the only partition of $8$ which qualifies while calculating $\textup{ssptd}_{o}(8, 3)$ is $8$. Hence $\textup{ssptd}_{o}(8, 3)=8$. Similarly, there is only one partition of $n-N=5$ which is to be considered while calculating $\textup{ssptd}_{o}(5, 3)$, and that is $5$ itself. Hence $\textup{ssptd}_{o}(5, 3)=5$. Thus, $\textup{ssptd}_{o}(8, 3)-\textup{ssptd}_{o}(5, 3)=8-5=3=a(8, 3)$, as guaranteed by Theorem \textup{\ref{chernbeck}}.
\end{example}

We also obtain other new results in addition to the ones stated above. These are given in the lemmas before proving some of the stated results. These include the partial fraction expansion of a finite analogue of Fine's function defined in \eqref{finefunction}, a finite analogue of the Rogers-Fine identity.

This paper is organized as follows. The preliminary results are collected in Section \ref{prelim}. The proofs of Theorem \ref{finmainabc}, its corollaries as well as a proof of Theorem \ref{fingenc} is given in Section \ref{prooffinmainabc}. The theory of $\textup{spt}(n, N)$ is developed in Section \ref{sptnNfunction}. In Section \ref{nscsection}, we give two proofs of our finite analogue of the Beck-Chern theorem and also develop the theory of $N_\textup{SC}(n, N)$. One of the proofs also gives an interesting result along the way (see \eqref{atw} below). A further corollary of Theorem \ref{fingenc}, which gives a nice relation between $d(n, N)$ and $\textup{lpt}(n, N)$, is stated and proved in Section \ref{furcor}. Section \ref{fingi} is devoted to proving Theorem \ref{fingithm} and its corollaries. In the same section, we also initiate the theory of $F_{N}(\a, \b;\tau)$, a finite analogue of Fine's function, by obtaining its partial fraction decomposition and a finite analogue of Rogers-Fine identity. We end the paper with some concluding remarks in Section \ref{cr}.

\section{Preliminaries}\label{prelim}
The $q$-binomial theorem is given by \cite[p.~17, Equation (2.2.1)]{gea1998}
\begin{equation}\label{q-binomial thm}
\sum_{n=0}^{\infty}\frac{(a)_{n}z^n}{(q)_n}=\frac{(az)_{\infty}}{(z)_{\infty}}\hspace{5mm}(|z|<1).
\end{equation}
From \cite[p.~36, Equations (3.3.6), (3.3.7)]{gea1998}, we have
\begin{align}
 (z)_N &= \sum_{n=0}^{N} \left[\begin{matrix} N \\ n \end{matrix}\right](-1)^{n}z^n q^{n(n-1)/2},\label{binomial0}\\
\frac{1}{(z)_{N}}&=\sum_{j=0}^{\infty}\left[\begin{matrix} N+j-1\\ j\end{matrix} \right]z^j.\nonumber
\end{align}
The $q$-Chu-Vandermonde identity is given by \cite[p. 354, II(6)]{gasperrahman}
\begin{equation}\label{Chu-Vandermonde}
{}_{2}\phi_{1}\left[ \begin{matrix} a,  q^{-M} \\
 d \end{matrix} \, ; q, q  \right] 
= \frac{(d/a)_M a^M}{(d)_M},
\end{equation}
where the basic hypergeometric series ${}_{r+1}\phi_{r}$ is defined as
\begin{equation*}\label{bhs}
{}_{r+1}\phi_{r}\left[\begin{matrix} a_1, a_2, \ldots, a_{r+1}\\
  b_1,  b_2, \ldots, b_{r} \end{matrix}\,; q,
z \right] :=\sum_{n=0}^{\infty} \frac{(a_1;q)_n (a_2;q)_n \cdots (a_{r+1};q)_n}{(q;q)_n (b_1;q)_n \cdots (b_{r};q)_n} z^n.
\end{equation*}
We also need \cite[Corollary 1.2]{overp}
\begin{equation}\label{qchuvan}
\sum_{n=0}^{N}\left[\begin{matrix} N \\ n \end{matrix}\right]\frac{(-1/a)_n(ac)^nq^{n(n+1)/2}}{(cq)_n}=\frac{(-acq)_{N}}{(cq)_N}.
\end{equation}
The partial fraction decomposition of $F(a,b;t)$, defined in \eqref{finefunction}, is given by \cite[p. 18, Equation (16.3)]{fine}
\begin{equation}\label{fine16.3}
F(a,b;t)=\frac{(aq)_{\infty}}{(bq)_{\infty}}\sum_{n=0}^{\infty}\frac{(b/a)_n}{(q)_n}\frac{(aq)^n}{1-tq^n}.	
\end{equation}
We also note \cite[p.~5, Equation (6.3)]{fine}
\begin{equation}\label{finetransform}
F(a, b;t)=\frac{1-b}{1-t}F\left(\frac{at}{b},t;b\right).
\end{equation}
The Heine transformation \cite[p.~19, Corollary 2.3]{gea1998} is given by
\begin{align}\label{heine}
{}_{2}\phi_{1}\left[ \begin{matrix} a,  b \\
 c \end{matrix} \, ; q, z  \right] = \frac{(b, a z; q)_{\infty }}{(c, z; q)_{\infty}} {}_{2}\phi_{1}\left[ \begin{matrix} \frac{c}{b} ,  z \\
 a z \end{matrix} \, ; q, b  \right],
\end{align}
whereas the finite Heine transformation, due to Andrews \cite[Theorem 2]{andfinheine}, is
\begin{equation}\label{fht}
_3 \phi _2
\left[\begin{matrix}
q^{-N}, &\alpha, &\beta \\
\gamma, & \frac{q^{1-N}}{\tau}
\end{matrix}; q, q\right]= \frac{(\beta, \alpha\tau;q)_N}{(\gamma,\tau;q)_N} {}_3 \phi _2
\left[\begin{matrix}
q^{-N}, & \frac{\gamma}{\beta}, &\tau\\
\alpha\tau, & \frac{q^{1-N}}{\beta}
\end{matrix}; q, q\right].
\end{equation}
We also need a corollary of \eqref{fht} given below \cite[Corollary 3, Equation 2.7]{andfinheine}.
\begin{equation}\label{corfht}
_3 \phi _2
\left[\begin{matrix}
q^{-N}, &\alpha, &\beta \\
\gamma, & \frac{q^{1-N}}{\tau}
\end{matrix}; q, q\right]= \frac{(\frac{\gamma}{\beta}, \beta\tau;q)_N}{(\gamma,\tau;q)_N}  {}_3 \phi _2
\left[\begin{matrix}
q^{-N}, & \frac{\alpha\beta\tau}{\gamma}, &\beta\\
\beta\tau, & \frac{\beta q^{1-N}}{\gamma}
\end{matrix}; q, q\right]
\end{equation}
A special case of \eqref{fht}, given by Rowell and Yee \cite[Lemma 3]{rowellyee09}, is
\begin{align}\label{rowyee}
\sum_{n=0}^{N} \left[\begin{matrix} N  \\ n \end{matrix}\right]  \frac{(-\alpha)_n (\tau q )^n}{(\tau q^{N+1-n})_n} = \frac{(-\alpha \tau q)_N}{(\tau q)_N}. 
\end{align} 
Watson's $q$-analogue of Whipple's theorem \cite[p.~360, III.18]{gasperrahman} is given by
\begin{multline}\label{watson87}
_8\phi_7\left[\begin{matrix} a,& q\sqrt{a},& -q\sqrt{a},& b, &
c, &d, & e, & q^{-N}\\
 &\sqrt{a}, & -\sqrt{a},& \df{a
q}{b}, & \df{a q}{c}, & \df{a q}{d}, & \df{a q}{e}, & a
q^{N+1}\end{matrix}\,; q,
 \df{a^2q^{N+2}}{bcde}\right] \\
=\df{(a q)_N\left(\df{a q}{de}\right)_N}
{\left(\df{a q}{d}\right)_N\left(\df{a q}{e}\right)_N}
\ _4\phi_3\left[\begin{matrix}\df{a q}{bc},d, e,q^{-N}\\ \df{a q}{b},
\df{a q}{c}, \df{de q^{-N}}{a}\end{matrix}\,; q, q\right].
\end{multline}
It is to be understood that the derivation of infinite series identities obtained by letting $N\to\infty$ in their respective finite analogues employs Tannery's theorem \cite[p.~136]{bromwich-1991a}.
\section{Proofs of Theorem \ref{finmainabc}, its corollaries and of Theorem \ref{fingenc}}\label{prooffinmainabc}
The two general identities, namely, Theorems \ref{finmainabc} and Theorem \ref{fingenc}, will be proved in this Section.

\begin{proof}[Theorem \textup{\ref{finmainabc}}][]
From \cite[p.~70, Equation (3.2.1)]{gasperrahman}, we have \\
\begin{equation*}
_4 \phi _3
\left[\begin{matrix}
q^{-N}, & A, &B, &C\\
D, & E, & \frac{ABCq^{1-N}}{DE}
\end{matrix}; q, q\right]= \frac{(\frac{E}{A})_N(\frac{DE}{BC})_N}{(E)_N(\frac{DE}{ABC})_N} {}_4 \phi _3
\left[\begin{matrix}
q^{-N}, & A, &\frac{D}{B}, &\frac{D}{C}\\
D, & \frac{DE}{BC}, & \frac{Aq^{1-N}}{E}
\end{matrix}; q, q\right].
\end{equation*}
Let $A=q$, $B=\frac{bq}{a},$ $C=cq,$ $D=bq$ and $E=cq^2$ in the above identity so that 
\begin{equation*}
\sum_{n=0}^{N}\frac{(q^{-N})_n(\frac{bq}{a})_n(1-cq)}{(bq)_n(1-cq^{n+1})(\frac{q^{1-N}}{a})_n}q^n=
\frac{(1-cq)(1-aq^N)}{(1-cq^{N+1})}\sum_{n=0}^{N}\frac{(q^{-N})_n(\frac{b}{c})_n}{(bq)_n(1-aq^n)(\frac{q^{-N}}{c})_n}q^n.
\end{equation*}
Employing 
\begin{align}
\left(\frac{q^{-N}}{c}\right)_n&=\frac{(-1)^n(cq^{N-n+1})_n q^{\frac{n(n-1)}{2}}}{c^n q^{Nn}},\label{e1}\\
\frac{\left(q^{-N}\right)_n}{\left({\frac{q^{-N}}{c}}\right)_n}&=\frac{\left(q^{N-n+1}\right)_n c^n}{\left(cq^{N-n+1}\right)_n}=\frac{(q)_N(cq)_{N-n}c^n}{(q)_{N-n}(cq)_N}\label{e2}
\end{align}
in the above equation, we see that
\begin{equation*}
\sum_{n=0}^{N}\frac{(q)_N(a)_{N-n}(\frac{bq}{a})_n}{(bq)_n(1-cq^{n+1})(q)_{N-n}(a)_{N+1}}a^n=
\sum_{n=0}^{N}\frac{(q)_N (cq)_{N-n}(\frac{b}{c})_n(cq)^n}{(1-aq^n)(q)_{N-n}(cq)_{N+1}(bq)_n}.
\end{equation*}
Now multiply both sides by $\frac{(a-b)(1-q^{N+1})}{1-b}$, and replace $n$ by $n-1$ in both the resulting sums to deduce
\begin{equation*}
\sum_{n=1}^{N+1}\frac{(q)_{N+1}(a)_{N-n+1}(\frac{b}{a})_{n}a^{n}}{(b)_n(1-cq^{n})(q)_{N-n+1}(a)_{N+1}}=
\sum_{n=1}^{N+1}\frac{(q)_{N+1} (cq)_{N-n+1}(\frac{b}{c})_{n-1}(cq)^{n-1}(a-b)}{(q)_{N-n+1}(cq)_{N+1}(b)_{n-1}(1-aq^{n-1})(1-bq^{n-1})}.
\end{equation*}
On the right-hand side of the above equation write
\begin{equation*}
\frac{(a-b)q^{n-1}}{(1-aq^{n-1})(1-bq^{n-1})}=\frac{aq^{n-1}}{1-aq^{n-1}}-\frac{bq^{n-1}}{1-bq^{n-1}},
\end{equation*}
and replace $N$ by $N-1$ on both sides to finally obtain \eqref{finmainabceqn}. 
\end{proof}
The special cases of Theorem \ref{fingGaravan} are now given. First given below is an identity of Yan and Fu \cite[Equation (4)]{yanfunanjing} whose limiting case $N\rightarrow \infty$ is the aforementioned identity of Yan and Fu \cite[p.~117]{yanfunanjing} rediscovered by Andrews, Garvan and Liang \cite[Theorem 3.5]{agl13}.
\begin{corollary}\label{fin_AGL} 
Let $N\in\mathbb{N}$. For $c\neq q^{-n}, 1\leq n\leq N$,
\begin{align*}
\sum_{n=1}^{N} \left[\begin{matrix} N\\n\end{matrix}\right] \frac{(-1)^{n-1}  q^{\frac{n(n+1)}{2}}}{ (1-cq^n)  } = \frac{1}{1-c}\left( 1- \frac{(q)_N}{(c q)_N} \right).
\end{align*}
\end{corollary}
\begin{proof}
Let $z=1$ in Theorem \ref{fingGaravan} and use the elementary identity $(c q)_{N-n}/(cq)_N = 1/(cq^{N+1-n})_n$ to obtain
\begin{align}
\sum_{n=1}^{N} \left[\begin{matrix} N\\n\end{matrix}\right] \frac{(-1)^{n-1}  q^{\frac{n(n+1)}{2}}}{ (1-cq^n)  } & = \frac{1}{c-1}\sum_{n=1}^{N}   \frac{(1/c)_n (cq)^n}{(q)_n} \frac{(q^{N+1-n})_n}{(c q^{N+1-n})_n} \nonumber\\
& = \frac{1}{c-1}\sum_{n=1}^{N}   \frac{(1/c)_n q^n}{(q)_n} \frac{(q^{-N})_n}{ \left(\frac{q^{-N}}{c}\right)_n}, \label{last}
\end{align}
where in the last step, we used \eqref{e2}. Now use the $q$-Chu-Vandermonde identity \eqref{Chu-Vandermonde} with $a = 1/c, M=N, $ and $d= q^{-N}/c$ to deduce
\begin{align}\label{applyChu}
\sum_{n=0}^{N}   \frac{(1/c)_n q^n}{(q)_n} \frac{(q^{-N})_n}{ \left(\frac{q^{-N}}{c}\right)_n}= \frac{(q^{-N})_N}{\left(\frac{q^{-N}}{c}\right)_N c^N } = \frac{(q)_N}{(cq)_{N}}. 
\end{align} 
Substituting \eqref{applyChu} in \eqref{last} completes the proof.
\end{proof}
Following is a finite analogue of Ramanujan's identity \eqref{Garvan's identity}.
\begin{corollary}\label{fin_Ram_Entry4}
Let $N\in\mathbb{N}$. For any $z \neq q^{-n}$ for $1 \leq n \leq N$, we have
\begin{align*}
\sum_{n=1}^{N}\left[\begin{matrix} N\\n\end{matrix}\right]\frac{(-1)^{n-1}z^nq^{\frac{n(n+1)}{2}}(q)_{n}}{(1-q^{n})(zq)_n} = \sum_{n=1}^N \frac{z q^n }{1- z q^n}. 
\end{align*}
\end{corollary}
\begin{proof}
Let $c=1$ in Theorem \ref{fingGaravan}.
\end{proof}
\begin{corollary}
Let $N\in\mathbb{N}$. For $c\neq q^{-n}, 1\leq n\leq N$,
\begin{equation}\label{fingen of sigma(q)}
\sum_{n=1}^{N}\left[\begin{matrix} N\\n\end{matrix}\right]\frac{q^{\frac{n(n+1)}{2}}(q)_{n}}{(1-cq^{n})(-q)_n}=
\frac{1}{c}\sum_{n=1}^{N}\left[\begin{matrix} N\\n\end{matrix}\right]\frac{(-q/c)_{n-1}(q)_{n} (cq)_{N-n} (cq)^{n}}{(-q)_{n}(cq)_{N}}.
\end{equation}
\end{corollary}
\begin{proof}
Let $z=-1$ in Theorem \ref{fingGaravan}.
\end{proof}
One may also specialize $z$ and $c$ in terms of $q^{m}, m\in\mathbb{N}\cup\{0\}$, to obtain other corollaries.

Note that when $c=0$, either side of \eqref{fingen of sigma(q)} gives a finite analogue of $\sigma(q)$ defined in \eqref{sigq}. We define this finite analogue by $\sigma(q,N)$ so that
\begin{equation}\label{sqn}
\sigma(q,N):=\sum_{n=0}^{N}\left[\begin{matrix} N\\n\end{matrix}\right] \frac{(q)_{n}q^{\frac{n(n+1)}{2}}}{(-q)_n}.
\end{equation}
The partition-theoretic interpretation of \eqref{sqn} is now given. Let $S_1$ be defined as in \eqref{S1}. Let $\vec{\pi}\in S_1$. Define $w_2(\vec{\pi}):=(-1)^{\#(\pi_1)+\textup{rank}(\pi_2)}$. Then 
\begin{equation*}
\sigma(q,N)=\sum_{m=0}^{\infty}\left(\sum_{n=1}^{N}\sum_{\substack{\vec{\pi}\in S_1\\ |\vec{\pi}|=m}}w_2(\vec{\pi})\right)q^m.
\end{equation*}
We now state and prove another corollary of Theorem \ref{finmainabc}.
\begin{corollary}\label{entry5fin}
For $N\in\mathbb{N}$,
\begin{equation*}
\sum_{n=1}^{N}\left[\begin{matrix} N\\n\end{matrix}\right]\frac{(q)_{n-1}(q)_n(zq)_{N-n}(zq)^n}{(zq)_n(zq)_{N}(1-cq^n)}=z\sum_{n=1}^{N}\left[\begin{matrix} N\\n\end{matrix}\right]\frac{(zq/c)_{n-1}(q)_n(cq)_{N-n}c^{n-1}q^n}{(zq)_n(cq)_{N}(1-zq^n)}.
\end{equation*}
\end{corollary}
\begin{proof}
Divide both sides of Theorem \ref{finmainabc} by $1-b/a$, let $b\to a$, then replace $a$ by $zq$ and simplify.
\end{proof}
This result, in turn, has two nice corollaries, the first of which is a generalization of an identity of Corteel and Lovejoy \cite[p.~1631]{overp}.
\begin{corollary}
Let $N\in\mathbb{N}$. Then,
\begin{align*}
\sum_{n=1}^{N} \frac{ q^n }{1- q^{2n}} = \sum_{n=1}^{N}\left[\begin{matrix} N\\n\end{matrix}\right] (-1)^{n-1} \frac{(-q)_{n-1}(-q)_{N-n}}{ (-q)_{N}} \frac{q^n}{1-q^n}.
\end{align*}
\end{corollary}
\begin{proof}
Let $z=-c=1$ in Corollary \ref{entry5fin}.
\end{proof}
\begin{corollary}
Let $N\in\mathbb{N}$ and $|q|<1$. Then,
\begin{align*}
\sum_{n=1}^{N}\left[\begin{matrix} N\\n\end{matrix}\right] \frac{(q)_{n-1}(q)_n(-q)_{N-n}}{ (-q)_{n-1}(-q)_{N}} \frac{(-q)^n}{1-q^{2n}}=\sum_{n=1}^{\infty} \frac{n (-q)^n(1-q^{Nn}) }{1- q^{n}}.
\end{align*}
\end{corollary}
\begin{proof}
Let $c=-z=1$ in Corollary \ref{entry5fin} and note that $\sum_{n=1}^{N}\frac{-q^n}{(1+q^n)^2}=\sum_{n=1}^{\infty} \frac{n (-q)^n(1-q^{Nn}) }{1- q^{n}}$.
\end{proof}

Before embarking upon the proof of Theorem \ref{fingenc}, which is responsible for much of the content in the sequel, we begin with a lemma.
\begin{lemma}\label{Second_Fines}
Let $N\in\mathbb{N}, |q|<1$ and $|cq|<1$. For $c\neq q^{-n}, 1\leq n\leq N$,
\begin{align}
\sum_{n=1}^{N} \frac{(-1)^{n-1}  q^{\frac{n(n+1)}{2}}}{ (1-cq^n) (q)_n (q)_{N-n} }  \left(\sum_{k=1}^n \frac{q^k}{1- q^k}  \right) = \sum_{n = 1}^{N}\left[\begin{matrix} N \nonumber \\ n \end{matrix}\right] \frac{q^{n(n + 1)}}{(q)_n (1 - q^n)} F(q^N, q^n; cq^n),
\end{align}
where $F(a,b;t)$ is the Fine's function defined in \eqref{finefunction}.
\end{lemma}
\begin{proof}
Using van Hamme's identity \eqref{hammeid}, we have
\begin{align}\label{ram}
&\sum_{n=1}^{N} \frac{(-1)^{n-1}  q^{\frac{n(n+1)}{2}}}{ (1-cq^n) (q)_n (q)_{N-n} }  \left(\sum_{k=1}^n \frac{q^k}{1- q^k}  \right) \nonumber\\
&= \sum_{n=1}^{N} \frac{(-1)^{n-1}  q^{\frac{n(n+1)}{2}}}{ (1-cq^n) (q)_{N-n} } \sum_{k=1}^{n}\frac{(-1)^{k-1} q^{k(k+1)/2}}{(q)_k (q)_{n-k} (1-q^k)}\nonumber\\
&=\sum_{k=1}^{N}\frac{(-1)^{k-1} q^{k(k+1)/2}}{(q)_k (1-q^k)}\sum_{n=k}^{N}\frac{(-1)^{n-1}  q^{\frac{n(n+1)}{2}}}{(1-cq^n) (q)_{N-n}(q)_{n-k}}\nonumber\\
&=\sum_{k=1}^{N}\frac{ q^{k(k+1)}}{(q)_k  (1-q^k)} \sum_{m=0}^{N-k} \frac{(-1)^m q^{\frac{m(m+1)}{2}+mk}}{(1-c q^{m+k})(q)_{N-m-k}(q)_m }\nonumber\\
&=\sum_{k=1}^{N}\frac{ q^{k(k+1)}}{(q)_k (q)_{N-k} (1-q^k)}\sum_{m=0}^{N-k}  \frac{ \left( q^{-(N-k)} \right)_mq^{(N+1)m} }{(q)_m(1-c q^{m+k}) },
\end{align}
where in the last step, we used \eqref{e1} with $N$ and $n$ replaced by $N-k$ and $m$ respectively.
Now use \eqref{fine16.3} with $a=q^N, b= q^k$ and $t=c q^k$ for the inner sum so that
{\allowdisplaybreaks\begin{align}\label{aftfine16.3}
&\sum_{k=1}^{N}\frac{ q^{k(k+1)}}{(q)_k (q)_{N-k} (1-q^k)}\sum_{m=0}^{N-k}  \frac{ \left( q^{-(N-k)} \right)_mq^{(N+1)m} }{(q)_m(1-c q^{m+k}) }\nonumber\\
&= \sum_{k=1}^{N}\frac{ q^{k(k+1)}}{(q)_k  (1-q^k)} \frac{(q^{k+1})_\infty}{(q)_{N-k}(q^{N+1})_\infty} F(q^{N}, q^k; c q^k) \nonumber \\
& =   \sum_{k=1}^{N} \left[\begin{matrix} N\\k\end{matrix}\right] \frac{ q^{k(k+1)}}{ (q)_k (1-q^k)}  F(q^{N}, q^k; c q^k).
\end{align}} 
Substituting \eqref{aftfine16.3} in \eqref{ram} completes the proof.
\end{proof}
\begin{proof}[Theorem \textup{\ref{fingenc}}][]
Differentiate both sides of Theorem \ref{fingGaravan} with respect to $z$ and then let $z=1$ so as to obtain
\begin{align}\label{Putting_z=1}
& \sum_{n=1}^{N} \frac{(-1)^{n-1}  q^{\frac{n(n+1)}{2}}}{ (1-cq^n) (q)_n (q)_{N-n} }  \left(n + \sum_{k=1}^n \frac{q^k}{1- q^k}  \right) \nonumber \\
&  =\frac{ 1}{c} \sum_{n=1}^N \frac{\left(q/c\right)_{n-1} (c q)_{N-n} (cq)^n }{(q)_n (c q)_N (q)_{N-n} } \nonumber \\
&  + \frac{1}{c}  \sum_{n=1}^N \frac{\left(q/c\right)_{n-1} (c q)_{N-n} (cq)^n }{(q)_n (c q)_N (q)_{N-n} } \left(-\sum_{k=1}^{n-1} \frac{q^k/c}{1- q^k/c} + \sum_{k=1}^n \frac{q^k}{1-q^k} \right)\nonumber\\
&=:S_1+S_2.
\end{align}
From Theorem \ref{fingGaravan} with $z=1$ and Corollary \ref{fin_AGL},
\begin{equation}\label{s1}
S_1=\frac{1}{(1-c)(q)_N}\left( 1- \frac{(q)_N}{(cq)_N}\right).
\end{equation}
A result of Guo and Zhang \cite[Corollary 3.1]{guozhang} states that if $n\geq 0$ and $0\leq m\leq n$,
\begin{align*}
&\sum_{k=0\atop k\neq m}^{n}\left[\begin{matrix} n\\k\end{matrix}\right]\frac{(q/x)_k(x)_{n-k}}{1-q^{k-m}}x^k\nonumber\\
&=(-1)^mq^{\frac{m(m+1)}{2}}\left[\begin{matrix} n\\m\end{matrix}\right](xq^{-m})_n\left(\sum_{k=0}^{n-1}\frac{xq^{k-m}}{1-xq^{k-m}}-\sum_{k=0\atop k\neq m}^{n}\frac{q^{k-m}}{1-q^{k-m}}\right).
\end{align*}
Letting $m=0$ in the above identity and simplifying leads to
\begin{align}\label{m0}
\sum_{k=1}^{n}\frac{q^k}{1-q^k}-\sum_{k=1}^{n-1}\frac{xq^k}{1-xq^k}=\frac{x}{1-x}-\frac{1}{(x)_n}\sum_{k=1}^{n}\left[\begin{matrix} n\\k\end{matrix}\right]\frac{(q/x)_k(x)_{n-k}x^k}{1-q^k}.
\end{align}
Employing \eqref{m0} with $x=1/c$ to simplify the expression in parenthesis in $S_2$ and using \eqref{s1} in the last step below, it is seen that
\begin{align}\label{s2aftid}
S_2&=\frac{1}{c}  \sum_{n=1}^N \frac{\left(q/c\right)_{n-1} (c q)_{N-n} (cq)^n }{(q)_n (c q)_N (q)_{N-n} } \left( \frac{1}{c-1} - \frac{1}{(1/c)_n} \sum_{k=1}^n \left[\begin{matrix} n\\k\end{matrix}\right] \frac{(qc)_k (1/c)_{n-k} c^{-k}}{1- q^k}\right)\nonumber\\
&=\frac{-1}{(1-c)^2(q)_N}\left( 1- \frac{(q)_N}{(cq)_N}\right) +S_{2}^{*},
\end{align}
where 
\begin{equation}\label{s2s}
S_2^{*}=\frac{1}{(1-c)(cq)_N}  \sum_{n=1}^N \frac{ (c q)_{N-n} (cq)^n }{(q)_n  (q)_{N-n}  } \sum_{k=1}^n \left[\begin{matrix} n\\k\end{matrix}\right] \frac{(qc)_k (1/c)_{n-k} c^{-k}}{1- q^k}.
\end{equation}
Next we simplify $S_2^{*}$. Note that
\begin{align*}
 S_2^{*}& =\frac{1}{(1-c)(cq)_N} \sum_{k=1}^N \frac{(cq)_k q^k}{(q)_k (1-q^k) } \sum_{j=0}^{N-k} \frac{(1/c)_j (cq)^j (cq)_{N-j-k}}{(q)_j(q)_{N-j-k}} \nonumber \\
 & =  \frac{1}{(1-c)(cq)_N} \sum_{k=1}^N \frac{(cq)_k (cq)_{N-k}q^k}{(q)_k (1-q^k)(q)_{N-k} } \sum_{j=0}^{N-k} \frac{(1/c)_j q^j \left(q^{-(N-k)}\right)_j}{(q)_j \left(q^{-(N-k)}/c\right)_j},
\end{align*}
where in the last step, we used \eqref{e2} with $N$ replaced by $N-k$ and $n$ replaced by $j$. Now apply $q$-Chu-Vandermonde identity \eqref{Chu-Vandermonde} with $a = 1/c, M= N-k$, and  $d= q^{-(N-k)}/c$ to see that
\begin{equation}\label{appl_Chu}
\sum_{j=0}^{N-k} \frac{(1/c)_j q^j \left(q^{-(N-k)}\right)_j}{(q)_j \left(q^{-(N-k)}/c\right)_j}=\frac{(q)_{N-k}}{(cq)_{N-k}},
\end{equation}
where we again employed \eqref{e2} with $N$ and $n$ both replaced by $N-k$. Substituting \eqref{appl_Chu} in \eqref{s2s} we derive
\begin{equation*}
S_2^{*}=\frac{1}{(1-c)(cq)_N} \sum_{k=1}^N \frac{(cq)_k q^k}{(q)_k (1-q^k) }
\end{equation*}
so that from \eqref{s2aftid},
\begin{align}\label{final_second term}
S_2=\frac{-1}{(1-c)^2(q)_N}\left( 1- \frac{(q)_N}{(cq)_N}\right) + \frac{1}{(1-c)(cq)_N} \sum_{k=1}^N \frac{(cq)_k q^k}{(q)_k (1-q^k) }, 
\end{align}
and hence from \eqref{Putting_z=1}, \eqref{s1} and \eqref{final_second term},
\begin{align*}
& \sum_{n=1}^{N} \frac{(-1)^{n-1}  q^{\frac{n(n+1)}{2}}}{ (1-cq^n) (q)_n (q)_{N-n} }  \left(n + \sum_{k=1}^n \frac{q^k}{1- q^k}  \right) \nonumber \\ 
& =\frac{-c}{(1-c)^2(q)_N}\left( 1- \frac{(q)_N}{(cq)_N}\right) + \frac{1}{(1-c)(cq)_N} \sum_{k=1}^N \frac{(cq)_k q^k}{(q)_k (1-q^k) }.
\end{align*}
Finally, invoke Lemma \ref{Second_Fines} in the above equation to arrive at \eqref{finanalogmaineqn}.
\end{proof}

\section{A finite analogue of Andrews' identity for $\textup{spt}(n)$}\label{sptnNfunction}

\begin{proof}[Theorem \textup{\ref{fin_spt_q_identity}}][]
Let $c \rightarrow 1$ on both sides of   \eqref{finanalogmaineqn} to get 
\begin{align}\label{limit_c=1}
{} & \frac{1}{(q)_N}\sum_{n = 1}^{N} \left[\begin{matrix} N \\ n \end{matrix}\right]\frac{(-1)^{n - 1}nq^{n(n + 1)/2}}{1 - q^n} + 
\sum_{n = 1}^{N}\left[\begin{matrix} N  \\ n \end{matrix}\right] \frac{q^{n(n + 1)}}{(q)_n (1 - q^n)} F(q^N, q^n; q^n)=L(q, N), 
\end{align}
where
\begin{equation*}
L(q, N):=\lim_{c\rightarrow 1} \left[ \frac{c}{(1 - c)^2 (q)_N}\left\{\frac{(q)_N}{(cq)_N} - 1\right\} + \frac{1}{(cq)_N (1 - c)}\sum_{n = 1}^{N}\frac{(cq)_n}{(q)_n}\frac{q^n}{1 - q^n} \right].
\end{equation*}
Let
\begin{align}\label{second term_Fine}
S:= \sum_{n = 1}^{N}\left[\begin{matrix} N  \\ n \end{matrix}\right] \frac{q^{n(n + 1)}}{(q)_n (1 - q^n)} F(q^N, q^n; q^n).
\end{align}
From \cite[p.~13, Equation (12.2)]{fine}, we know that
\begin{align*}
F(a,b;t)= \frac{1}{1-t} \sum_{n=0}^{\infty} \frac{(b/a)_n (-at)^n q^{n(n+1)/2}}{(bq)_n (tq)_n}.
\end{align*}
Using the above representation with $a=q^{N}$ and $b=t=q^n$ in \eqref{second term_Fine} and then employing \eqref{e1} with $n$ replaced by $j$ and $N$ replaced by $N-n$ in the second step below, we have
\begin{align}\label{S}
S&= \sum_{n = 1}^{N}\left[\begin{matrix} N  \\ n \end{matrix}\right] \frac{q^{n(n + 1)}}{(q)_n (1 - q^n)^2} \sum_{j=0}^{\infty} \frac{\left( q^{-(N-n)} \right)_j (-1)^jq^{(N+n)j} q^{j(j+1)/2}}{ (q^{n+1})_j^2}\nonumber\\
&= \sum_{n = 1}^{N}\left[\begin{matrix} N  \\ n \end{matrix}\right] \frac{q^{n(n + 1)}}{(q)_n (1 - q^n)^2} \sum_{j=0}^{N-n} \frac{ q^{j^2 + 2n j} (q)_{N-n}}{ (q)_{N-n-j}(q^{n+1})_j^2} \nonumber \\
& = (q)_N \sum_{n = 1}^{N} \frac{q^{n(n + 1)}}{(q)_n^2 (1 - q^n)^2} \sum_{j=0}^{N-n} \frac{ q^{j^2 + 2n j} }{ (q)_{N-n-j}(q^{n+1})_j^2} \nonumber \\
& =  (q)_N \sum_{n = 1}^{N} \frac{q^{n}}{ (1 - q^n)^2} \sum_{j=0}^{N-n} \frac{ q^{(j+n)^2} }{ (q)_{N-n-j}(q)_{n+j}^2}\nonumber\\
& =  (q)_N \sum_{n = 1}^{N} \frac{q^{n}}{ (1 - q^n)^2} \sum_{j=n}^{N} \frac{ q^{j^2} }{ (q)_{N-j}(q)_{j}^2}\nonumber\\
&=(q)_N \sum_{j=1}^{N} \frac{ q^{j^2} }{ (q)_{N-j}(q)_{j}^2}\sum_{n = 1}^{j} \frac{q^{n}}{ (1 - q^n)^2}.
\end{align}
To evaluate $L(q, N)$, first let $\tau=1$ and $\a=-c$ in \eqref{rowyee} to obtain
\begin{align*}
 \sum_{n=0}^{N}   \frac{(c)_n  q^n}{(q)_n} = \frac{(c q)_N}{( q)_N}. 
\end{align*}
Using this in the second step below, we see that
{\allowdisplaybreaks\begin{align}
L&= \lim_{c\rightarrow 1} \frac{1}{(cq)_N}\left[ \frac{c }{(1 - c)^2 }\left(1 - \frac{(cq)_N}{(q)_N} \right) + \frac{1}{ (1 - c)}\sum_{n = 1}^{N}\frac{(cq)_n}{(q)_n}\frac{q^n}{1 - q^n} \right]\nonumber\\
& = \frac{1}{(q)_N} \lim_{c\rightarrow 1} \left[ \frac{c }{(c-1) } \sum_{n=1}^{N} \frac{(cq)_{n-1} q^n}{(q)_n} + \frac{1}{ (1 - c)}\sum_{n = 1}^{N}\frac{(cq)_n}{(q)_n}\frac{q^n}{1 - q^n} \right] \nonumber \\
& = \frac{1}{(q)_N} \lim_{c\rightarrow 1} \left[ \frac{1}{(1-c) } \sum_{n=1}^{N} \frac{(cq)_{n-1} q^n}{(q)_n} \left( -c +  \frac{1-cq^n}{1-q^n}\right) \right] \nonumber \\
& =   \frac{1}{(q)_N}  \sum_{n=1}^N \frac{q^n}{(1-q^n)^2}\label{6s}\\
& = \frac{1}{(q)_N}\left(\sum_{n=1}^\infty \frac{q^n}{(1-q^n)^2} - \sum_{n=N+1}^\infty \frac{q^n}{(1-q^n)^2}\right)\nonumber\\
& = \frac{1}{(q)_N}\left(\sum_{m=1}^\infty m \sum_{n=1}^\infty (q^m)^n - \sum_{m=1}^\infty m \sum_{n=N+1}^\infty (q^m)^n\right) \nonumber \\
& = \frac{1}{(q)_N}\sum_{m=1}^\infty \frac{m q^m(1-q^{Nm})}{1-q^m}.\label{final_limit_c=1}
\end{align}}
Now substitute \eqref{S} and \eqref{final_limit_c=1} in \eqref{limit_c=1} to complete the proof.
\end{proof}
As remarked in Section \ref{new}, we now show that Theorem \ref{fin_spt_q_identity} is nothing but the analytical version of Theorem \ref{fin_spt_identity}. We first need a lemma.
\begin{lemma} \label{q_N id}
We have
\begin{equation*}
(q)_N = \sum_{n=1}^{N+1}
\left[\begin{matrix} N+1 \\ n \end{matrix}\right](-1)^{n-1}nq^{n(n-1)/2}.
\end{equation*}
\end{lemma}
\begin{proof}
Replace $N$ by $N+1$ in \eqref{binomial0} and then differentiate both sides with respect to $z$ so as to obtain
\begin{equation*}
 -(z)_{N+1}\sum_{k=0}^{N}\frac{q^k}{1-zq^k} = \sum_{n=0}^{N+1} \left[\begin{matrix} N+1 \\ n \end{matrix}\right](-1)^{n}nz^{n-1} q^{n(n-1)/2}.
\end{equation*}
Now let $z \rightarrow 1$ and observe that the left hand side becomes $-(1-q)(1-q^2)\cdots(1-q^N)$, and thus we have the result.
\end{proof}
\textbf{First proof of Theorem \ref{fin_spt_identity}.}
We begin by proving
\begin{align}\label{gen_finite_spt}
\sum_{n=1}^\infty \textup{spt}(n,N) q^n = \frac{1}{(q)_N}\sum_{n = 1}^{N} \left[\begin{matrix} N \\ n \end{matrix}\right]\frac{(-1)^{n - 1}nq^{n(n + 1)/2}}{1 - q^n},
\end{align}
where $\textup{spt}(n,N)$ is given in Definition 1. By a simple combinatorial argument, one can see that
\begin{align}\label{gen_spt(n,N)}
\sum_{n=1}^\infty \textup{spt}(n,N) q^n  = \sum_{n=1}^N \frac{q^n}{(1-q^n)^2(1- q^{n+1})\cdots (1-q^N)}.
\end{align}
Let $A(q, N)$ and $B(q, N)$ denote the right-hand sides of \eqref{gen_spt(n,N)} and \eqref{gen_finite_spt} respectively. It suffices to show that $A(q, N)=B(q, N)$.  We apply induction on $N$ to prove this. Note that  
\begin{equation*}
A(q, 1)=B(q, 1)= \frac{q}{(1-q)^2}.
\end{equation*}
By induction hypothesis, assume that $A(q, N)=B(q, N)$. We then show that $A(q, N+1)=B(q, N+1)$. This is done by showing that both $A(q, N)$ and $B(q, N)$ satisfy the recurrence relation
\begin{align*}
f(q, N+1)= \frac{f(q, N)}{1-q^{N+1}} + \frac{q^{N+1}}{(1-q^{N+1})^2}.
\end{align*} 
Clearly $A(q, N)$ satisfies the above recurrence relation. To prove that $B(q, N)$ does so too, we first separate the $(N+1)$-th term in $B(q, N+1)$ so that
\begin{align*}
B(q, N+1) = \sum_{n = 1}^{N} \frac{(-1)^{n - 1}nq^{n(n + 1)/2}}{(q)_n (1 - q^n) (q)_{N-n}} \frac{1}{(1 - q^{N+1-n})} + \frac{(-1)^N (N+1) q^{(N+1)(N+2)/2}}{(q)_{N+1} (1- q^{N+1})}.
\end{align*}
Thus,
\begin{align*}
B(q, N+1) - \frac{B(q, N)}{1-q^{N+1}} & = \sum_{n = 1}^{N} \frac{(-1)^{n - 1}nq^{n(n - 1)/2}}{(q)_n  (q)_{N+1-n}} \frac{q^{N+1}}{(1 - q^{N+1})} + \frac{(-1)^N (N+1) q^{(N+1)(N+2)/2}}{(q)_{N+1} (1- q^{N+1})} \nonumber \\
& =  \frac{q^{N+1}}{(1 - q^{N+1})} \sum_{n = 1}^{N+1} \frac{(-1)^{n - 1}nq^{n(n - 1)/2}}{(q)_n  (q)_{N+1-n}} \nonumber \\
& =\frac{q^{N+1}}{(1 - q^{N+1}) (q)_{N+1}} \sum_{n = 1}^{N+1} \left[\begin{matrix} N+1 \\ n \end{matrix}\right] (-1)^{n - 1}nq^{n(n - 1)/2} \nonumber \\
&=\frac{q^{N+1}}{(1-q^{N+1})^2},
\end{align*}
by an application of Lemma \ref{q_N id}. This proves \eqref{gen_finite_spt}.

Next, we have
\begin{align}
\sum_{j = 1}^{N} \left[\begin{matrix} N \\ j \end{matrix}\right] \frac{q^{j^2}}{(q)_j} \sum_{k=1}^j \frac{q^k}{(1-q^k)^2 }&=\frac{1}{2} \frac{d^2}{dz^2}\left( \sum_{j=1}^N \left[\begin{matrix} N \\ j \end{matrix}\right] \frac{q^{j^2} (q)_j }{(z q)_j (z^{-1} q)_j} \right)_{z=1}, \label{2drc1}\\
\frac{1}{(q)_N} \sum_{n=1}^{\infty} \frac{n q^n(1- q^{N n})}{1- q^n}&=\frac{1}{(q)_N}\sum_{n=1}^{N}\frac{q^n}{(1-q^n)^2}=\frac{1}{2} \frac{d^2}{dz^2}\left( \frac{(q)_N}{(z q)_N(z^{-1}q)_N} \right)_{z=1},\label{2drc2}
\end{align}
both of which are easily proved by routine differentiation techniques, with the first equality in \eqref{2drc2} resulting from \eqref{6s} and \eqref{final_limit_c=1}. Thus from \eqref{gen_finite_spt}, \eqref{2drc1}, \eqref{2drc2} and Theorem \ref{fin_spt_q_identity}, we deduce that
\begin{align}\label{3}
\sum_{n=1}^\infty \textup{spt}(n,N) q^n = \frac{1}{2} \frac{d^2}{dz^2}\left( \frac{(q)_N}{(z q)_N(z^{-1}q)_N} \right)_{z=1} - \frac{1}{2} \frac{d^2}{dz^2}\left( \sum_{j=1}^N \left[\begin{matrix} N \\ j \end{matrix}\right] \frac{q^{j^2} (q)_j }{(z q)_j (z^{-1} q)_j} \right)_{z=1}.
\end{align}
Now from Theorem \ref{frgfthm} and \eqref{finrankmom} and the fact that any odd finite rank moment is equal to zero, we see that
\begin{equation}\label{4}
\frac{d^2}{dz^2}\left( \sum_{j=1}^N \left[\begin{matrix} N \\ j \end{matrix}\right] \frac{q^{j^2} (q)_j }{(z q)_j (z^{-1} q)_j} \right)_{z=1}=\sum_{n=1}^{\infty}N_{2,N}(n)q^n.
\end{equation}
Similarly from Theorem \ref{fcgfthm} and \eqref{fincrankmom},
\begin{equation}\label{5}
\frac{d^2}{dz^2}\left( \frac{(q)_N}{(z q)_N(z^{-1}q)_N} \right)_{z=1}=\sum_{n=1}^{\infty}M_{2,N}(n)q^n.
\end{equation}
Thus from \eqref{3}, \eqref{4} and \eqref{5}, we arrive at \eqref{1}. 

To prove \eqref{2}, first note that
\begin{align*}
\sum_{n=1}^{\infty} \frac{n q^n(1- q^{N n})}{1- q^n}&=\sum_{n=1}^{\infty}\sum_{m=1}^{\infty} nq^{mn}-\sum_{n=1}^{\infty}\sum_{m=N+1}^{\infty} nq^{mn}.
\end{align*}
Clearly,
\begin{equation*}
\sum_{n=1}^{\infty}\sum_{m=1}^{\infty} nq^{mn}=\sum_{k=1}^{\infty}\left(\sum_{n|k}n\right)q^k.
\end{equation*}
As for the second sum, let $mn=k$. Then $k=mn>Nn$. Hence
\begin{equation*}
\sum_{n=1}^{\infty}\sum_{m=N+1}^{\infty} nq^{mn}=\sum_{k=1}^{\infty}\left(\sum_{n|k\atop n<k/N}n\right)q^k,
\end{equation*}
and so
\begin{align*}
\sum_{n=1}^{\infty} \frac{n q^n(1- q^{N n})}{1- q^n}=\sum_{k=1}^{\infty}\sigma(k, N)q^k,
\end{align*}
where $\sigma(k, N)$ is defined in \eqref{skn}. Hence
\begin{align}\label{7}
\frac{1}{(q)_N}\sum_{n=1}^{\infty} \frac{n q^n(1- q^{N n})}{1- q^n}=\sum_{n=1}^{\infty}\left(\sum_{j=0}^{n-1}p(j, n)\sigma(n-j,N)\right)q^n.
\end{align}
Therefore from Theorem \ref{fin_spt_q_identity}, \eqref{gen_finite_spt}, \eqref{4} and \eqref{7}, we arrive at \eqref{2}.
\qed\\

As mentioned in Section \ref{new}, we now offer another proof of Theorem \ref{fin_spt_identity} closely following Andrews' proof of \eqref{Andrews_spt0} in \cite{andrews08}.\\

\textbf{Second proof of Theorem \ref{fin_spt_identity}.}
Let $a\to 1, d=e^{-1}=z$ followed by $b, c\to\infty$ in Watson's $q$-analogue of Whipple's theorem \eqref{watson87} so as to obtain after simplification
\begin{align*}
&1+(1-z)(1-z^{-1})\sum_{n=1}^{N}\frac{(-1)^n(1+q^n)q^{\frac{n(3n+1)}{2}}(q)_{N}^2}{(1-zq^n)(1-z^{-1}q^n)(q)_{N-n}(q)_{N+n}}\nonumber\\
&=\frac{(q)_N^{2}}{(zq)_N(z^{-1}q)_N}\sum_{n=0}^{N}\frac{(z)_{n}(z^{-1})_{n}q^n}{(q)_n}.
\end{align*}
Applying the operator $\left.\frac{d^2}{dz^2}\right|_{z=1}$ on both sides while using \cite[Equation (2.1)]{andrews08}, we deduce after simplification
\begin{equation*}
\frac{1}{(q)_N}\sum_{n=1}^N\frac{(q)_{n-1}q^n}{(1-q^n)}=\frac{1}{(q)_N}\sum_{n=1}^N\frac{q^n}{(1-q^n)^2}+\sum_{n=1}^{N}\frac{(-1)^n(1+q^n)q^{\frac{n(3n+1)}{2}}(q)_{N}}{(1-q^n)^2(q)_{N-n}(q)_{N+n}}.
\end{equation*}
Now \eqref{gen_spt(n,N)}, \eqref{2drc2} and \eqref{5} imply
\begin{align*}
\frac{1}{(q)_N}\sum_{n=1}^N\frac{(q)_{n-1}q^n}{(1-q^n)}&=\sum_{n=1}^{\infty}\textup{spt}(n, N)q^n,\nonumber\\
\frac{1}{(q)_N}\sum_{n=1}^N\frac{q^n}{(1-q^n)^2}&=\frac{1}{2}\sum_{n=1}^{\infty}M_{2,N}(n)q^n.
\end{align*}
Hence it suffices to show that
\begin{equation*}
\sum_{n=1}^{N}\frac{(-1)^n(1+q^n)q^{\frac{n(3n+1)}{2}}(q)_{N}}{(1-q^n)^2(q)_{N-n}(q)_{N+n}}=-\frac{1}{2}\sum_{n=0}^{\infty}N_{2,N}(n)q^n.
\end{equation*}
From ,
\begin{align*}
&\frac{1}{2}\sum_{n=0}^{\infty}N_{2,N}(n)q^n\nonumber\\
&= \frac{d^2}{dz^2}\left[\frac{(1-z)}{2}\sum_{n=0}^{N}\left[\begin{matrix} N \\ n \end{matrix}\right]\frac{(q)_nq^{n^2}}{(z)_{n+1}(z^{-1}q)_n} \right]_{z=1}\nonumber\\
&= \frac{d}{dz}\left[\frac{(1-z)}{2}\frac{d}{dz}\sum_{n=0}^{N}\left[\begin{matrix} N \\ n \end{matrix}\right]\frac{(q)_nq^{n^2}}{(z)_{n+1}(z^{-1}q)_n}- \frac{1}{2}\sum_{n=0}^{N}\left[\begin{matrix} N \\ n \end{matrix}\right]\frac{(q)_nq^{n^2}}{(z)_{n+1}(z^{-1}q)_n}\right ]  _{z=1}\nonumber\\
&=\frac{d}{dz}\bigg[\frac{1}{2(1-z)(q)_N}+\frac{1}{2}\sum_{n=1}^{N}\left[\begin{matrix} N \\ n \end{matrix}\right]\frac{(-1)^n (q)_{n}q^{\frac{n(3n+1)}{2}}}{(q)_{n+N}}\left(\frac{(1-z)q^n}{(1-z q^n)^2}+ \frac{1-z}{(z-q^n)^2} \right)\nonumber\\
&\quad-\frac{1}{2}\sum_{n=0}^{N}\left[\begin{matrix} N \\ n \end{matrix}\right]\frac{(q)_n q^{n^2}}{(z)_{n+1}( z^{-1}q)_n} \bigg]_{z=1}\nonumber\\
&=\bigg\{\frac{1}{2}\sum_{n=1}^{N}\left[\begin{matrix} N \\ n \end{matrix}\right]\frac{(-1)^n (q)_{n}q^{\frac{n(3n+1)}{2}}}{(q)_{n+N}}\left(\frac{(2q^n-1-zq^n)q^n}{(1-z q^n)^3}+ \frac{z-2+q^n}{(z-q^n)^3} \right)\nonumber\\
&\quad-\frac{1}{2}\sum_{n=1}^{N}\left[\begin{matrix} N \\ n \end{matrix}\right]\frac{(-1)^n (q)_{n}q^{\frac{n(3n+1)}{2}}}{(q)_{n+N}}\left(\frac{q^n}{(1-z q^n)^2}+ \frac{1}{(z-q^n)^2} \right)\bigg\}_{z=1}\nonumber\\
&=\sum_{n=1}^{N}\frac{(-1)^{n+1}(1+q^n)q^{\frac{n(3n+1)}{2}}(q)_{N}}{(1-q^n)^2(q)_{N-n}(q)_{N+n}},
\end{align*}
where we used \eqref{frgfbil} in the third step. This completes the proof.
\qed\\

\begin{proof}[Theorem \textup{\ref{asyest}}][]
Let $F(q, N)$ denote the right-hand side of \eqref{gen_spt(n,N)}. It is clear that $F(q, N)$ has a pole of order $N+1$ at $q=1$. Moreover, if $\zeta_{i}$ is a primitive $i$th root of unity then $F(q, N)$ has a pole of order $\left[\frac{N}{i}\right]+1$ at $\zeta_{i}$. Thus $F(q, N)$ has the partial fraction decomposition
\begin{align}\label{pfd}
F(q, N)= \frac{A}{(1-q)^{N+1}} + \frac{B}{(1-q)^N} + \cdots,
\end{align}
where $A$ and $B$ are some constants. As we shall see, the main contribution will come from the first term. By binomial theorem,
\begin{align*}
\frac{1}{(1-q)^{N+1}} & = \sum_{n=0}^\infty \frac{(N+1)_n}{n!} q^n, \\
\frac{1}{(1-q)^N} & = \sum_{n=0}^\infty \frac{(N)_n}{n!} q^n,
\end{align*}
where $(N)_n:=N(N+1)\cdots(N+n-1)$ denotes the rising factorial.
Using these series expansions in \eqref{pfd} and then comparing the coefficient of $q^n$, we have
\begin{equation}\label{partial}
\textup{spt}(n,N) =  \frac{A\cdot (N+1)_n}{n!} + \frac{B\cdot (N)_n}{n!}  + \cdots. 
\end{equation} 
Note that as $n\to\infty$,
\begin{align}\label{asymp_1st_term}
 \frac{ (N+1)_n}{n!} =  \frac{ (N+n)!}{N! n!} & = \frac{(n+1)(n+2)\cdots(n+N)}{N!} \nonumber \\
 & = \frac{n^N}{N!} + O(n^{N-1}).
\end{align}
Similarly,
\begin{equation}\label{asymp_2nd_term}
\frac{ (N)_n}{n!} = \frac{n^{N-1}}{(N-1)!} + O(n^{N-2}).
\end{equation} 
Use \eqref{asymp_1st_term} and \eqref{asymp_2nd_term} in \eqref{partial} to obtain 
\begin{equation*}
\textup{spt}(n, N) =  A \frac{ n^N}{N!} + O(n^{N-1}). 
\end{equation*}
Now multiply $F(q, N)$ by $(1-q)^{N+1}$ and then take limit $q\rightarrow 1$ to deduce $A = \frac{1}{N!}$. This completes the proof of the theorem. 
\end{proof}

\section{A finite analogue of the Beck-Chern theorem and properties of $N_\textup{SC}(n, N)$}\label{nscsection}

We start this section with a corollary of Theorem \ref{fingenc} which motivates us to study finite analogues of an important class of partitions.
\begin{corollary}\label{nscc-1}
Let $N\in\mathbb{N}$. We have
\begin{equation}
\begin{aligned}
{} & \frac{1}{(q)_N}\sum_{n = 1}^{N} \left[\begin{matrix} N \\ n \end{matrix}\right]\frac{(-1)^{n - 1}nq^{n(n + 1)/2}}{1 + q^n} + 
\frac{1}{2(-q)_N}\sum_{k=1}^{N}\left[\begin{matrix} N \\ k \end{matrix} \right] \frac{q^{k(k+1)/2}}{(1-q^k)}
\left(\frac{(-q)_k}{(q)_k} - 1 \right) \\
& = \ \frac{1}{4(q)_N}\left\{1 - \frac{(q)_N}{(-q)_N}\right\} + \frac{1}{2(-q)_N}\sum_{n = 1}^{N}\frac{(-q)_n}{(q)_n}\frac{q^n}{1 - q^n}. \label{c=-1}
\end{aligned}
\end{equation}
\end{corollary}
\begin{proof}
By putting $c=-1$ in Theorem \ref{fingenc}, we clearly get the right hand side and the first term on the left hand side of equation \eqref{c=-1} above. Thus we need only show that
\begin{align}
\sum_{n = 1}^{N}\left[\begin{matrix} N \nonumber \\ n \end{matrix}\right] \frac{q^{n(n + 1)}}{(q)_n (1 - q^n)} F\left(q^N, q^n; -q^n\right)=\frac{1}{2(-q)_N}\sum_{k=1}^{N}\left[\begin{matrix} N \\ k \end{matrix} \right] \frac{q^{k(k+1)/2}}{(1-q^k)}\left(\frac{(-q)_k}{(q)_k} - 1 \right).
\end{align}
Letting $a=q^{N}, b=-t=q^n$ in \eqref{finetransform} gives
\begin{equation*}
F(q^N, q^n; -q^n) = \frac{1 - q^n}{1 + q^n}F(-q^N, -q^n; q^n). 
\end{equation*}
Thus, employing the above equation in the first step below and then \eqref{fine16.3} in the second step, we see that
\begin{align*}
&\sum_{n = 1}^{N}\left[\begin{matrix} N \nonumber \\ n \end{matrix}\right] \frac{q^{n(n + 1)}}{(q)_n (1 - q^n)} F\left(q^N, q^n; -q^n\right)\nonumber\\
&=\sum_{n=1}^{N}\left[\begin{matrix} N \\ n \end{matrix}\right] \frac{q^{n(n + 1)}}{(q)_n (1 + q^n)} F(-q^N, -q^n; q^n)\nonumber\\
&= \sum_{n=1}^{N}\left[\begin{matrix} N \\ n \end{matrix}\right]\frac{q^{n(n+1)}}{(q)_n (1 + q^n)}
\frac{(-q^{N+1})_{\infty}}{(-q^{n+1})_{\infty}}\sum_{j=0}^{\infty}\frac{(q^{n-N})_j}{(q)_j}\frac{(-q^{N+1})^j}{1-q^{n + j}} \nonumber \\ 
{} &= \ \sum_{n=1}^{N}\left[\begin{matrix} N \\ n \end{matrix} \right]\frac{(-q)_n q^{n(n+1)}}{(-q)_N (q)_{n} (1 + q^n)}
\sum_{j=0}^{\infty}\frac{(q)_{N-n} q^{\frac{j(j+1)}{2} + nj}}{(q)_{N-n-j}(q)_j(1 - q^{n + j})}\nonumber\\
&= \sum_{n=1}^{N}\left[\begin{matrix} N \\ n \end{matrix} \right]\frac{(-q)_n q^{\frac{n(n+1)}{2}}}{(-q)_N (q)_{n} (1 + q^n)}
\sum_{j=0}^{\infty}\frac{(q)_{N-n} q^{(n+j)(n+j+1)/2}}{(q)_{N-n-j}(q)_j(1 - q^{n + j})}\label{2.18}
\nonumber\\
&=\sum_{n=1}^{N}\left[\begin{matrix} N \\ n \end{matrix} \right]\frac{(-q)_n q^{\frac{n(n+1)}{2}}}{(-q)_N (q)_{n} (1 + q^n)}
 \sum_{k=n}^{\infty} \frac{(q)_{N-n} q^{k(k+1)/2}}{(q)_{N-k}(q)_{k-n}(1 - q^{k})}\nonumber\\
& =\sum_{k=1}^{\infty} \frac{q^{k(k+1)/2}}{(q)_{N-k}(1-q^k)}\sum_{n=1}^{\min(k, N)} \left[\begin{matrix} N \\ n \end{matrix} \right] \frac{(-q)_n}{(-q)_N}\frac{q^{n(n+1)/2 }(q)_{N-n}}{(q)_{k-n}(q)_n (1 + q^n)} \nonumber \\
&= \frac{1}{2(-q)_N}\sum_{k=1}^{N}\left[\begin{matrix} N \\ k \end{matrix} \right] \frac{q^{k(k+1)/2}}{(1-q^k)}
\sum_{n=1}^{k}\left[\begin{matrix} k \\ n \end{matrix} \right] \frac{(-1)_n}{(q)_n}q^{n(n+1)/2} \nonumber \\
&= \frac{1}{2(-q)_N}\sum_{k=1}^{N}\left[\begin{matrix} N \\ k \end{matrix} \right] \frac{q^{k(k+1)/2}}{(1-q^k)}
\left(\frac{(-q)_k}{(q)_k} - 1 \right),
\end{align*}
where the evaluation of the inner sum in the last step follows from \eqref{qchuvan} with $a=c=1$. This completes the proof.
\end{proof}
As we shall see, \eqref{c=-1} leads us to study the finite analogues of what are called $S$ - partitions and self-conjugate $S$ - partitions defined by Andrews, Garvan and Liang in \cite[pp. 199--200]{agl13}.

Let $V$ denote the set of vector partitions, that is, $V=\mathcal{D}\times\mathcal{P}\times\mathcal{P}$, where $\mathcal{P}$ denotes the set of partitions and $\mathcal{D}$ denotes the set of partitions into distinct parts. For a positive integer $N$, let $S_N$ denote the following set of vector partitions:
\begin{equation*}
S_N:=\{\vec{\pi}=(\pi_1, \pi_2, \pi_3)\in V: 1\leq s(\pi_1)<\infty, \hspace{1mm}s(\pi_1)\leq\min(s(\pi_2), s(\pi_3)) \hspace{1mm} \text{and} \hspace{1mm} l(\pi_1), l(\pi_2), l(\pi_3) \leq N\}.
\end{equation*}
Let the vector partition $\vec{\pi}$ belonging to the set $S_N$ be called an $S_N$-partition.
Let $w_{\textup{SC}}(\vec{\pi})=(-1)^{\#(\pi_1)-1}$ and define the involution map $\imath: S_N\to S_N$ by
\begin{equation*}
\imath(\vec{\pi})=\imath(\pi_1, \pi_2, \pi_3)=\imath(\pi_1, \pi_3, \pi_2).
\end{equation*}
Define an $S_N$-partition $\vec{\pi}=(\pi_1, \pi_2, \pi_3)$ to be a self-conjugate $S_N$-partition if it is a fixed point of $\imath$, that is, if and only if $\pi_2=\pi_3$. Let $N_{\textup{SC}}(n, N)$ denote the number of self-conjugate $S_N$-partitions counted according to the weight $w_{\textup{SC}}$, that is,
\begin{equation}\label{nscofnN}
N_{\textup{SC}}(n, N)=\sum_{\vec{\pi}\in S_N, |\vec{\pi}|=n \atop \imath(\vec{\pi})=\vec{\pi}}w_{\textup{SC}}(\vec{\pi}).
\end{equation}
We first find the generating function for $N_{\textup{SC}}(n, N)$.
\begin{theorem}\label{nscN}
Let $N\in\mathbb{N}$. We have
\begin{equation}\label{GFNSCofnN}
 \sum_{n=1}^{N}\frac{q^n(q^{n+1})_{N-n}}{(q^{2n};q^2)_{N-n+1}}=\sum_{n=1}^{\infty}N_{\textup{SC}}(n, N)q^n.
\end{equation}
\end{theorem}
\begin{proof}
For a fixed $n, 1 \leq n \leq N$, the numerator generates partitions $\pi_1$ into \textit{distinct parts} with smallest part $s(\pi_1)=n$, largest part $l(\pi_1)\leq N$ and counted with weight $(-1)^{\#(\pi_1)-1}.$ The denominator generates partitions into parts lying in $\{2n, 2n+2,\dots, 2N\}$, or equivalently, two identical partitions $\pi_2$ and $\pi_3$ with parts in $\{n, n+1, \dots, N\}$. In other words, $\pi_2$ and $\pi_3$ satisfy
$s(\pi_1)=n \leq s(\pi_2)=s(\pi_3)$ and $l(\pi_2), l(\pi_3) \leq N.$ So, \eqref{GFNSCofnN} generates precisely those partitions of $S_N$ with $\pi_2 = \pi_3$ and counted with weight $(-1)^{\#(\pi_1)-1} = w_{\textup{SC}}(\vec{\pi}),$ where $\vec{\pi} = (\pi_1, \pi_2, \pi_3).$ This completes the proof.
\end{proof}
In the theorem below, we obtain another representation for the generating function of $N_{\textup{SC}}(n, N)$.
\begin{theorem}\label{NSCinduction}
\begin{equation}
 \sum_{n=1}^{N}\frac{q^n(q^{n+1})_{N-n}}{(q^{2n};q^2)_{N-n+1}} = \frac{1}{(q)_N}\sum_{n = 1}^{N} \left[\begin{matrix} N \\ n \end{matrix}\right]\frac{(-1)^{n - 1}nq^{n(n + 1)/2}}{1 + q^n}.
\end{equation}
\end{theorem}
\begin{proof}
Let $C(q, N)$ and $D(q, N)$ respectively denote the left- and right-hand sides of the above equation. Observe that $C(q, 1) = D(q, 1) = \frac{q}{1-q^2}$. We show both $C(q, N)$ and $D(q, N)$ satisfy the same recurrence relation
\begin{equation}
f(q, N+1)=\frac{f(q, N)}{1+q^{N+1}}+\frac{q^{N+1}}{1-q^{2N+2}},
\end{equation}
whence we will be done. To that end, note that
 \begin{align*}
C(q, N+1)  &= \ \sum_{n=1}^{N}\left\{\frac{q^n (q^{n+1})_{N-n}}{(q^{2n};q^2)_{N-n+1}}\cdot\frac{1-q^{N+1}}{1-q^{2N+2}} \right\}+ 
\frac{q^{N+1}}{1-q^{2N+2}}\\
& = \ \frac{C(q, N)}{1+q^{N+1}} + \frac{q^{N+1}}{1-q^{2N+2}}.
\end{align*}
Next, separating the $(N+1)$-th term of $D(q, N+1)$ in the first step below, we see that
\begin{align*}
&D(q, N+1) - \frac{D(q, N)}{1+q^{N+1}}\nonumber\\
&= \sum_{n=1}^{N}\frac{(-1)^{n-1}nq^{n(n+1)/2}}{(q)_n(1+q^n)(q)_{N-n+1}} + \frac{(-1)^{N}(N+1)q^{(N+1)(N+2)/2}}{(1+q^{N+1})(q)_{N+1}} - \frac{1}{1+q^{N+1}}\sum_{n=1}^{N}\frac{(-1)^{n-1}nq^{n(n+1)/2}}{(q)_n(1+q^n)(q)_{N-n}} \\
& = \ \frac{q^{N+1}}{1 + q^{N+1}}\sum_{n=1}^{N}\frac{(-1)^{n-1}nq^{n(n-1)/2}}{(q)_{n}(q)_{N-n+1}} + \frac{(-1)^{N}(N+1)q^{(N+1)(N+2)/2}}{(1+q^{N+1})(q)_{N+1}}\nonumber\\
&= \ \frac{q^{N+1}}{(1+q^{N+1})(q)_{N+1}}\sum_{n=1}^{N+1}
\left[\begin{matrix} N+1 \\ n \end{matrix}\right](-1)^{n-1}nq^{n(n-1)/2} \\
& = \ \frac{q^{N+1}}{1-q^{2N+2}},
\end{align*}
by Lemma \ref{q_N id}. Thus, $C(q, N)=D(q, N)$ for all positive integers $N$.
\end{proof}
We now state Corollary \ref{nscc-1} in the form that will be used in the sequel. This is a finite analogue of Corollary 2.12 of \cite{dixitmaji18}.
\begin{corollary}\label{c=-1finalcor}
Let $N\in\mathbb{N}$. Then
\begin{equation*}
\begin{aligned}\label{c=-1final}
{} & (q)_N\sum_{n = 1}^{\infty}N_{\textup{SC}}(n, N)q^n + 
\frac{(q)_N}{2(-q)_N}\sum_{k=1}^{N}\left[\begin{matrix} N \\ k \end{matrix} \right] \frac{q^{k(k+1)/2}}{(1-q^k)}
\left(\frac{(-q)_k}{(q)_k} - 1 \right) \\
{} & = \ \frac{1}{4}\left\{1 - \frac{(q)_N}{(-q)_N}\right\} + \frac{(q)_N}{2(-q)_N}\sum_{n = 1}^{N}\frac{(-q)_n}{(q)_n}\frac{q^n}{1 - q^n}. 
\end{aligned}
\end{equation*}
\end{corollary}
\begin{proof}
Multiply both sides of Corollary \ref{nscc-1} by $(q)_N$ and employ Theorems \ref{nscN} and \ref{NSCinduction}.
\end{proof}
We next prove Theorem \ref{chernbeck}. To do this, however, we first require two lemmas. The first one below gives the closed-form evaluation of a special case of Fine's function.
\begin{lemma}\label{Fineevaluation}
Let $N\in\mathbb{N}$. Let $F(a, b; t)$ be defined in \eqref{finefunction}. Then
\begin{equation*}
F\left(\frac{-1}{q}, -q^N; q\right) = \frac{1+q^N}{1-q^N}[(-1)_N - 1]. 
\end{equation*}
\end{lemma}
\begin{proof}
Using \eqref{finetransform} in the first step below, and \eqref{e1} with $n$ and $N$ respectively replaced by $m$ and $N-1$ in the second step and simplifying, we see that
{\allowdisplaybreaks\begin{align*}
 F\left(\frac{-1}{q}, -q^N; q\right)&=\frac{1+q^N}{1-q}\sum_{m=0}^{\infty}\frac{(q^{-(N-1)})_m}{(q^2)_m}(-q^N)^m\nonumber\\
 {} &= \frac{1+ q^N}{1- q^N}\sum_{m=0}^{\infty}\frac{(q)_{N}}{(q)_{N-1-m}}\frac{q^{m(m+1)/2}}{(q)_{m+1}}\\
 {} &= \frac{1+ q^N}{1- q^N}\sum_{m=1}^{N} \left[ \begin{matrix} N \\ m \end{matrix} \right]q^{m(m-1)/2}\\
 {} &= \frac{1+ q^N}{1- q^N}[(-1)_N - 1],
\end{align*}}
where we used \eqref{binomial0} with $z=-1$ in the last step.
\end{proof}
Next we state another lemma of which two proofs are given. The first proof led us to the right-hand side starting from the left. Once the identity was known, we obtained a shorter proof by induction. We give both since the first one also gives a new identity along the way (see \eqref{atw} below).
\begin{lemma}\label{sumexpliciteval}
  Let $m, N\in\mathbb{N}$. Then
  \begin{equation*}\label{sumexplicitevaleqn}
   \sum_{n=0}^{m-1} \frac{(-1)_{n} q^{n}}{(-q^N)_{n+1}} = \frac{1}{1-q^N}\left\{\frac{(-1)_m}{(-q^N)_m} - 1\right\}.
  \end{equation*}
\end{lemma}
\begin{remark}
The limiting case $N\to\infty$ of the above result is well-known. To see this, first equate the two expressions for the generating function of the number of partitions of a number into distinct parts with largest part $\leq m-1$, that is,
\begin{equation}\label{dis}
(-q)_{m-1}=1+\sum_{n=1}^{m-1}(-q)_{n-1}q^n,
\end{equation}
where the expression on the right-hand side is easily obtained by fixing $n$ to be the largest part in a partition in the aforementioned collection. Then multiplying both sides of \eqref{dis} by $2$ and subtracting $1$ from both sides leads to
\begin{align}\label{tbu}
(-1)_{m}-1=1+\sum_{n=1}^{m-1}(-1)_nq^n=\sum_{n=0}^{m-1}(-1)_nq^n.
\end{align}
\end{remark}
\noindent
\textbf{First proof of Lemma \ref{sumexpliciteval}:} 
Use \eqref{qchuvan} with $a=q^{-N}, c=-q^N$ in the second step below to see that
\begin{align}\label{ae1}
\sum_{n=0}^{m-1} \frac{(-1)_{n} q^{n}}{(-q^N)_{n+1}}&=\sum_{n=0}^{m-1} \frac{(-1)_{n} q^{n}}{(q)_n}\frac{(q)_n}{(-q^N)_{n+1}}\nonumber\\
&=\sum_{n=0}^{m-1}\frac{(-1)_{n} q^{n}}{(q)_{n}}\sum_{k=0}^{n}(-1)^{k}\left[\begin{matrix} n\\k\end{matrix}\right]\frac{q^{k(k+1)/2}}{1+q^{N+k}}\nonumber\\
&=\sum_{k=0}^{m-1}\frac{(-1)^kq^{k(k+1)/2}}{(q)_k(1+q^{N+k})}\sum_{n=k}^{m-1}\frac{(-1)_n}{(q)_{n-k}}q^n\nonumber\\
&=\sum_{k=0}^{m-1}\frac{(-1)^kq^{k(k+3)/2}(-1)_k}{(q)_k(1+q^{N+k})}\sum_{n=0}^{m-1-k}\frac{(-q^k)_n}{(q)_{n}}q^n\nonumber\\
&=\sum_{k=0}^{m-1}\frac{(-1)^kq^{k(k+3)/2}(-1)_k}{(q)_k(1+q^{N+k})}\frac{(-q^{k})_{m-k}}{(q)_{m-1-k}(1+q^k)},
\end{align}
where in the last step we applied an identity due to Fu and Lascoux \cite[Equation (1.2)]{fulascoux}, namely,
\begin{equation*}
\sum_{j=0}^{m}\frac{(z)_j}{(q)_j}(-xq)^{j}=\frac{(z)_{m+1}}{(q)_{m}}\sum_{j=0}^{m}\left[\begin{matrix} m\\j\end{matrix}\right]\frac{(-xq)^{j}(-1/x)_{j}}{(1-zq^j)}
\end{equation*}
with $z=-q^{k}$ and $x=-1$ and $m$ replaced by $m-1-k$. Thus from \eqref{ae1} and using the elementary identity $\sum_{j=1}^{\infty}\frac{1-z^j}{1-z}x^j=\frac{x}{(1-zx)(1-x)}, |x|<1,$ in the third step below, we have
\begin{align}
&\sum_{n=0}^{m-1} \frac{(-1)_{n} q^{n}}{(-q^N)_{n+1}}\nonumber\\
&=\frac{(-1)_m}{(q)_{m-1}}\sum_{k=0}^{m-1}\left[\begin{matrix} m-1\\k\end{matrix}\right](-1)^kq^{k(k+1)/2}\frac{q^k}{(1+q^{N+k})(1+q^k)}\nonumber\\
&=\frac{(-q)_{m-1}}{(q)_{m-1}(1+q^{N})}-\frac{(-1)_m}{(q)_{m-1}}\sum_{k=1}^{m-1}\left[\begin{matrix} m-1\\k\end{matrix}\right](-1)^kq^{k(k+1)/2}\sum_{j=1}^{\infty}\frac{1-q^{Nj}}{1-q^{N}}(-q^k)^j\nonumber\\
&=\frac{(-q)_{m-1}}{(q)_{m-1}(1+q^{N})}+\frac{(-1)_m}{(q)_{m-1}(1-q^{N})}\sum_{j=1}^{\infty}(-1)^{j-1}(1-q^{Nj})\sum_{k=1}^{m-1}\left[\begin{matrix} m-1\\k\end{matrix}\right](-1)^kq^{k(j+1)+\frac{k(k-1)}{2}}\nonumber\\
&=\frac{(-q)_{m-1}}{(q)_{m-1}(1+q^{N})}+\frac{(-1)_m}{(q)_{m-1}(1-q^{N})}\sum_{j=1}^{\infty}(-1)^{j-1}(1-q^{Nj})\left[(q^{j+1})_{m-1}-1\right],
\end{align}
where we used \eqref{binomial0} in the last step. By an application of \eqref{q-binomial thm} in the second step below,
\begin{align}\label{atw0}
&\sum_{n=0}^{m-1} \frac{(-1)_{n} q^{n}}{(-q^N)_{n+1}}\nonumber\\
&=\frac{(-q)_{m-1}}{(q)_{m-1}(1+q^{N})}-\frac{(-1)_m}{(q)_{m-1}(1-q^{N})}\bigg\{\sum_{j=1}^{\infty}(-1)^j\frac{(q)_{m+j-1}-(q)_j}{(q)_j}\nonumber\\
&\qquad\qquad\qquad\qquad\qquad\qquad-(q)_{m-1}\sum_{j=1}^{\infty}\frac{(q^m)_j}{(q)_j}(-q^{N})^{j}+\sum_{j=1}^{\infty}(-q^{N})^{j}\bigg\}\nonumber\\
&=\frac{(-q)_{m-1}}{(q)_{m-1}(1+q^{N})}-\frac{(-1)_m}{(q)_{m-1}(1-q^{N})}\bigg\{\sum_{j=1}^{\infty}(-1)^j\frac{(q)_{m+j-1}-(q)_j}{(q)_j}\nonumber\\
&\qquad\qquad\qquad\qquad\qquad\qquad-(q)_{m-1}\left(\frac{(-q^{N+m})_{\infty}}{(-q^{N})_{\infty}}-1\right)-\frac{q^{N}}{1+q^{N}}\bigg\}.
\end{align}
Letting $N\to\infty$, employing \eqref{tbu} and simplifying results in
\begin{equation}\label{atw}
\sum_{j=1}^{\infty}(-1)^j\frac{(q)_{m+j-1}-(q)_j}{(q)_j}=\frac{1}{2}+\frac{(q)_{m-1}}{(-1)_m}-(q)_{m-1},
\end{equation}
which is an interesting result that we get along the way. Now substitute \eqref{atw} in \eqref{atw0} and simplify to finally obtain \eqref{sumexplicitevaleqn}.
\qed\\

\noindent
\textbf{Second proof of Lemma \ref{sumexpliciteval}:} Let $P(q, m)$ and $Q(q, m)$ respectively denote the left- and right-hand sides of \eqref{sumexplicitevaleqn}. Note that $P(q, 1)=Q(q, 1)=1/(1+q^{N})$. It is readily seen that 
\begin{equation*}
 P(q, m+1)-P(q, m) = \frac{(-1)_m q^m}{(-q^N)_{m+1}}.
\end{equation*}
Now
\begin{align*}
 {} Q(q, m+1) - Q(q, m) &= \frac{1}{1-q^N}\left\{\frac{(-1)_{m+1}}{(-q^N)_{m+1}} - \frac{(-1)_{m}}{(-q^N)_{m}}\right\}\\
 {} &= \frac{1}{(1-q^N)}\frac{(-1)_{m}}{(-q^N)_{m+1}}\left\{ (1+q^m) - (1+q^{N+m})\right\} \\
{} &= \frac{(-1)_m q^m}{(-q^N)_{m+1}}
= P(q, m+1)-P(q, m).
\end{align*}
By induction on $m$, the proof of the lemma is complete.
\qed

\noindent
Armed with the above two lemmas, we now give two proofs of our finite analogue of the Beck-Chern theorem.

\noindent
\textbf{First proof of Theorem \ref{chernbeck}.}
Let $a(n, N)$ and $\textup{ssptd}_{o}(n, N)$ be as defined in the hypotheses of the theorem. Then by the definition of $a(n, N)$,
\begin{align}\label{anNgen}
\sum_{n=1}^{\infty} a(n, N)q^n=\sum_{n=1}^{N} \frac{q^n (-q)_{n-1}}{1-q^n}.
\end{align}
Also, if  $\mathfrak{D}(n, N)$ denotes the set of partitions $\pi$ of $n$ into distinct parts with $l(\pi) - s(\pi)\leq N-1$, then
\begin{align}
&\sum_{n=1}^{\infty} \textup{ssptd}_{o}(n, N)q^n\nonumber\\
 &= \frac{1}{2}\left(\sum_{n=1}^{\infty}\sum_{\pi \in \mathfrak{D}(n, N)} s(\pi)q^n -\sum_{n=1}^{\infty}\sum_{\pi \in \mathfrak{D}(n, N)} (-1)^{\#(\pi)}s(\pi)q^n \right)\nonumber\\
&=\frac{1}{2}\left(\sum_{n=1}^{\infty} nq^n (-q^{n+1})_{N-1} + \sum_{n=1}^{\infty} nq^n (q^{n+1})_{N-1}\right)\label{ssptdogf}\\
&=\frac{1}{2}\left((-q)_{N-1}\sum_{n=1}^{\infty} \frac{nq^n (-q^N)_{n}}{(-q)_n}+(q)_{N-1}\sum_{n=1}^{\infty} \frac{nq^n (q^N)_{n}}{(q)_n}\right)\nonumber\\
&=\frac{1}{2}\left[(-q)_{N-1}\left.\left\{z \frac{\partial}{\partial z}\sum_{n=1}^{\infty}\frac{z^n (-q^N)_{n}}{(-q)_n}\right\}\right|_{z=q}+(q)_{N-1}\left.\left\{z \frac{\partial}{\partial z}\sum_{n=1}^{\infty}\frac{z^n (q^N)_{n}}{(q)_n}\right\}\right|_{z=q}\right]\nonumber\\
&=: \frac{1}{2}\left[(-q)_{N-1}. \left.G(z, q, N)\right|_{z=q} + (q)_{N-1}. \left.H(z, q, N)\right|_{z=q}\right].\label{ssptdo}
\end{align}
By the $q$-binomial theorem \eqref{q-binomial thm},
\begin{align*}
H(z, q, N)=z \frac{\partial}{\partial z}\frac{(q^{N}z)_{\infty}}{(z)_{\infty}}=z \frac{\partial}{\partial z}\frac{1}{(z)_N} = \frac{z}{(z)_N}\sum_{r=0}^{N-1}\frac{q^r}{1-zq^r}
\end{align*}
so that
\begin{align}\label{ssptdohzqn}
(1-q^{N})(q)_{N-1}\cdot\left.H(z, q, N)\right|_{z=q}=\sum_{r=1}^{N}\frac{q^r}{1-q^r}.
\end{align}
Next, invoking \eqref{heine}, we observe that
\begin{align}\label{gzqnpri}
G(z, q, N)&= z \frac{\partial}{\partial z} {}_{2}\phi_{1}\left[ \begin{matrix} -q^N, & q \\ & -q \end{matrix} \, ; q, z  \right]\nonumber\\
&=z \frac{\partial}{\partial z} \frac{(q)_{\infty}(-q^N z)_{\infty}}{(-q)_{\infty}(z)_{\infty}} {}_{2}\phi_{1}\left[ \begin{matrix} -1, & z \\ & -q^{N}z \end{matrix} \, ; q, q  \right]\nonumber\\
&=\frac{(q)_{\infty}}{(-q)_{\infty}} z \frac{\partial}{\partial z}\sum_{n=0}^{\infty} \frac{(-1)_n (-zq^{N+n})_{\infty}}{(zq^n)_{\infty} (q)_n}q^n \nonumber\\
&= \frac{(q)_{\infty}}{(-q)_{\infty}} \sum_{n=0}^{\infty}\frac{(-1)_nq^n}{(q)_n}\frac{(-zq^{N+n})_{\infty}}{(zq^n)_{\infty}} \sum_{k=0}^{\infty} \left\{\frac{zq^{N+n+k}}{1+zq^{N+n+k}}+\frac{zq^{n+k}}{1-zq^{n+k}}\right\}.
\end{align}
Letting $n+k=m$ in the last expression in \eqref{gzqnpri} and employing Lemma \ref{sumexpliciteval} in the second step, we have
\begin{align}\label{4sum}
(1-q^{N})(-q)_{N-1}\cdot\left.G(z, q, N)\right|_{z=q}&=(1-q^{N})\sum_{m=1}^{\infty}\left\{\frac{q^{N+m}}{1+q^{N+m}}+\frac{q^{m}}{1-q^{m}}\right\}\sum_{n=0}^{m-1}\frac{(-1)_n}{(-q^{N})_{n+1}}q^n\nonumber\\
&=\sum_{m=1}^{\infty}\left\{\frac{q^{N+m}}{1+q^{N+m}}+\frac{q^{m}}{1-q^{m}}\right\}\left(\frac{(-1)_m}{(-q^{N})_m}-1\right)\nonumber\\
&=\sum_{m=1}^{\infty}\frac{q^{N+m}}{1+q^{N+m}}\frac{(-1)_m}{(-q^{N})_m}- \sum_{m=1}^{\infty}\frac{q^{N+m}}{1+q^{N+m}}\nonumber\\
&\quad+ \sum_{m=1}^{\infty} \frac{q^{m}}{1-q^{m}}\frac{(-1)_m}{(-q^N)_m}- \sum_{m=1}^{\infty}\frac{q^{m}}{1-q^{m}}.
\end{align}
Recalling the Fine's function defined in \eqref{finefunction}, the first sum in the above equation can be evaluated as follows:
\begin{align}\label{B4firstpart}
 \sum_{m=1}^{\infty} \frac{q^{N+m}}{1+q^{N+m}}\frac{(-1)_m}{(-q^N)_m}& = \frac{q^N}{1+q^N}\left(F\left(\frac{-1}{q}, -q^N; q\right) - 1\right)\nonumber\\
&= \frac{(-1)_Nq^N}{1-q^N} - \frac{2q^N}{1-q^{2N}},
\end{align}
where in the second step, we invoked Lemma \ref{Fineevaluation}. Moreover, the second and the fourth sums in \eqref{4sum} combine together to give
\begin{align}\label{B4secondthird}
- \sum_{m=1}^{\infty}\frac{q^{N+m}}{1+q^{N+m}}- \sum_{m=1}^{\infty}\frac{q^{m}}{1-q^{m}}=-\sum_{m=1}^{N}\frac{q^m}{1-q^m} - \sum_{m=N+1}^{\infty}\frac{2q^m}{1-q^{2m}}.
\end{align}
Thus from \eqref{4sum}, \eqref{B4firstpart} and \eqref{B4secondthird},
\begin{align}\label{4sum1}
&(1-q^{N})(-q)_{N-1}\cdot\left.G(z, q, N)\right|_{z=q}\nonumber\\
&= \frac{q^N}{1-q^N}(-1)_N - \sum_{m=N}^{\infty}\frac{2q^m}{1-q^{2m}} + \sum_{m=1}^{\infty}\frac{q^{m}}{1-q^{m}}\frac{(-1)_m}{(-q^N)_m}-\sum_{m=1}^{N}\frac{q^m}{1-q^m}.
\end{align}
Hence from \eqref{ssptdo}, \eqref{ssptdohzqn} and \eqref{4sum1},
\begin{align*}
(1-q^{N})\sum_{n=1}^{\infty} \textup{ssptd}_{o}(n, N)q^n= \frac{1}{2}\left\{\sum_{m=1}^{\infty}\frac{q^{m}}{1-q^{m}}\frac{(-1)_m}{(-q^N)_m} - \sum_{m=N}^{\infty}\frac{2q^m}{1-q^{2m}}+ \frac{q^N}{1-q^N}(-1)_N\right\}.
\end{align*}
Note that $(-1)_N q^N/(1-q^N)$ is the $N$th term in 
\begin{equation*}\label{ssptdo2}
 \sum_{m=1}^{N} \frac{q^m}{1-q^m}(-1)_m = 2\sum_{n=1}^{\infty} a(n, N)q^n.
\end{equation*}
Thus we will be done provided we show 
\begin{equation}\label{ssptdaltproof}
\sum_{m=1}^{\infty}\frac{q^{m}}{1-q^{m}}\frac{(-1)_m}{(-q^N)_m} - \sum_{m=N}^{\infty}\frac{2q^m}{1-q^{2m}} = \sum_{m=1}^{N-1} \frac{q^m}{1-q^m}(-1)_m,
\end{equation}
for, then
\begin{equation*}
(1-q^{N})\sum_{n=1}^{\infty} \textup{ssptd}_{o}(n, N)q^n=\sum_{n=1}^{\infty} a(n, N)q^n,
\end{equation*}
which implies \eqref{bcgen}. We now prove \eqref{ssptdaltproof} by induction on $N$. 

Let $R(q, N)$ and $S(q, N)$ respectively denote the left- and right-hand sides of \eqref{ssptdaltproof}. Note that $R(q, 1)=S(q, 1)=0$. Suppose $R(q, N)=S(q, N)$ for some positive integer $N$. Now
{\allowdisplaybreaks\begin{align*}
{} R(q, N+1) &= \sum_{m=1}^{\infty}\frac{q^{m}}{1-q^{m}}\frac{(-1)_m}{(-q^{N+1})_m} - \sum_{m=N+1}^{\infty}\frac{2q^m}{1-q^{2m}}\\
{} &= \sum_{m=1}^{\infty}\frac{q^{m}}{1-q^{m}}\frac{(-1)_m}{(-q^N)_m}\frac{1+q^N}{1+q^{N+m}}- \sum_{m=N}^{\infty}\frac{2q^m}{1-q^{2m}} + \frac{2q^N}{1-q^{2N}}\\
{} &= \sum_{m=1}^{\infty}\frac{q^{m}}{1-q^{m}}\frac{(-1)_m}{(-q^N)_m}\frac{1+q^{N+m}+q^N(1-q^m)}{1+q^{N+m}} - \sum_{m=N}^{\infty}\frac{2q^m}{1-q^{2m}} + \frac{2q^N}{1-q^{2N}}\\
{} &= \sum_{m=1}^{\infty}\frac{q^{m}}{1-q^{m}}\frac{(-1)_m}{(-q^N)_m} - \sum_{m=N}^{\infty}\frac{2q^m}{1-q^{2m}} + q^N \sum_{m=1}^{\infty}\frac{q^{m}}{1+q^{N+m}}\frac{(-1)_m}{(-q^N)_m} + \frac{2q^N}{1-q^{2N}}\\
{} &= \sum_{m=1}^{N-1} \frac{q^m}{1-q^m}(-1)_m + \frac{q^N}{1+q^N}\left(F\left(\frac{-1}{q}, -q^N; q\right) - 1\right) + \frac{2q^N}{1-q^{2N}}\nonumber\\
&= \sum_{m=1}^{N-1} \frac{q^m}{1-q^m}(-1)_m + \frac{q^N}{1+q^N}\left\{ \frac{1+q^N}{1-q^N}[(-1)_N - 1] - 1\right\} + \frac{2q^N}{1-q^{2N}}\\
 {} &= \sum_{m=1}^{N-1} \frac{q^m}{1-q^m}(-1)_m + \frac{q^N}{1-q^N}(-1)_N\nonumber\\
&= \sum_{m=1}^{N} \frac{q^m}{1-q^m}(-1)_m\nonumber\\
&=S(q, N+1),
\end{align*}}
where in the fifth and sixth steps, we respectively used the induction hypothesis and Lemma \ref{Fineevaluation}. This proves \eqref{ssptdaltproof} and hence completes the proof.
\qed\\

\noindent
\textbf{Second proof of Theorem \ref{chernbeck}.}
From \eqref{anNgen} and \eqref{ssptdogf}, it suffices to show that $U(q, N)=V(q, N)$, where
\begin{align*}
U(q, N)&:=\sum_{n=1}^{N} \frac{q^n (-q)_{n-1}}{1-q^n},\\
V(q, N)&:=\frac{1-q^N}{2}\left(\sum_{n=1}^{\infty} nq^n (-q^{n+1})_{N-1} + \sum_{n=1}^{\infty} nq^n (q^{n+1})_{N-1}\right).
\end{align*}
Note that $U(q, 1)=V(q, 1)$. Assume that  $U(q, N)=V(q, N)$. Observe that
\begin{equation}\label{inductionequality}
 U(q, N+1) - U(q, N) = \frac{q^{N+1} (-q)_N}{1-q^{N+1}}.
\end{equation}
We will be done if we can show that $U(q, N+1) - U(q, N)=V(q, N+1) - V(q, N)$. Now
\begin{align}\label{inductionmiddleexp}
V(q, N+1) - V(q, N) &= \frac{1}{2}(1-q^{N+1})\left\{\sum_{n=1}^{\infty} nq^n (-q^{n+1})_{N} + \sum_{n=1}^{\infty} nq^n (q^{n+1})_{N}\right\}\nonumber\\
&\quad- \frac{1}{2}(1-q^N)\left\{\sum_{n=1}^{\infty} nq^n (-q^{n+1})_{N-1} + \sum_{n=1}^{\infty} nq^n (q^{n+1})_{N-1}\right\}\nonumber\\
&= \frac{1}{2}q^{N+1}\sum_{n=1}^{\infty}  nq^{n-1} (-q^{n})_{N} - \frac{1}{2}q^{N+1}\sum_{n=1}^{\infty}  nq^n (-q^{n+1})_{N} \nonumber\\
&\quad+ \frac{1}{2}q^{N+1}\sum_{n=1}^{\infty}  nq^{n-1} (q^{n})_{N} - \frac{1}{2}q^{N+1}\sum_{n=1}^{\infty}  nq^{n} (q^{n+1})_{N}.
\end{align}
Consider the first two sums on the right hand side of \eqref{inductionmiddleexp}, that is,
\begin{align*}
 &\frac{1}{2}q^{N+1}\left\{\sum_{n=1}^{\infty}  nq^{n-1} (-q^{n})_{N} - \sum_{n=1}^{\infty}  nq^n (-q^{n+1})_{N}\right\}\nonumber\\
&=\frac{1}{2}q^{N+1}\left\{\sum_{n=1}^{\infty}  nq^{n-1} (-q^{n})_{N} - \sum_{n=1}^{\infty}  (n+1)q^n (-q^{n+1})_{N} + \sum_{n=1}^{\infty} q^n(-q^{n+1})_N\right\}\nonumber\\
&=\frac{1}{2}q^{N+1}\sum_{n=0}^{\infty} q^n(-q^{n+1})_N.
\end{align*}
Similarly, the last two sums on the right hand side of \eqref{inductionmiddleexp} combine to give
 $\frac{q^{N+1}}{2} \sum_{n=0}^{\infty} q^n(q^{n+1})_N$ so that
\begin{equation}\label{BNfinalexp}
 V(q, N+1) - V(q, N) = \frac{q^{N+1}}{2}\sum_{n=0}^{\infty} q^n \left\{(-q^{n+1})_N + (q^{n+1})_N \right\}.
\end{equation}
From \eqref{inductionequality} and \eqref{BNfinalexp}, it suffices to show
\begin{equation}\label{finalequality}
(-q)_N = \frac{1-q^{N+1}}{2} \sum_{n=0}^{\infty} q^n \left\{(-q^{n+1})_N + (q^{n+1})_N \right\}.
\end{equation}
To that end, the right-hand side of \eqref{finalequality} can be simplified in the following way:
\begin{align*}
&\frac{(1-q^{N+1})}{2} \sum_{n=0}^{\infty} q^n \left\{(-q^{n+1})_N + (q^{n+1})_N \right\}\nonumber\\
&=\frac{1}{2}\sum_{n=0}^{\infty}q^n\left\{(-q^{n+1})_N + (q^{n+1})_N \right\} - \frac{1}{2}\sum_{n=0}^{\infty}q^{n+N+1}\left\{(-q^{n+1})_N + (q^{n+1})_N \right\}\nonumber\\
&=\frac{1}{2}\sum_{n=0}^{\infty}(1+q^n)(-q^{n+1})_N -  \frac{1}{2}\sum_{n=0}^{\infty}(-q^{n+1})_N 
- \frac{1}{2}\sum_{n=0}^{\infty}(1- q^n)(q^{n+1})_N + \frac{1}{2}\sum_{n=0}^{\infty}(q^{n+1})_N \nonumber\\
&\quad- \frac{1}{2}\sum_{n=0}^{\infty}(1+q^{n+N+1})(-q^{n+1})_N + \frac{1}{2}\sum_{n=0}^{\infty}(-q^{n+1})_N + \frac{1}{2}\sum_{n=0}^{\infty}(1-q^{n+N+1})(q^{n+1})_N - \frac{1}{2}\sum_{n=0}^{\infty}(q^{n+1})_N \nonumber\\
&=\frac{1}{2}\sum_{n=0}^{\infty}(-q^{n})_{N+1}  - 
 \frac{1}{2}\sum_{n=0}^{\infty}(q^{n})_{N+1} -  \frac{1}{2}\sum_{n=0}^{\infty}(-q^{n+1})_{N+1} 
+  \frac{1}{2}\sum_{n=0}^{\infty}(q^{n+1})_{N+1} \nonumber\\
&= \left\{\frac{1}{2}\sum_{n=0}^{\infty}(-q^{n})_{N+1} -\frac{1}{2}\sum_{n=0}^{\infty}(-q^{n+1})_{N+1}\right\}
- \left\{\frac{1}{2}\sum_{n=0}^{\infty}(q^{n})_{N+1} - \frac{1}{2}\sum_{n=0}^{\infty}(q^{n+1})_{N+1}\right\}\nonumber\\
&=(-q)_{N},
\end{align*}
since the sums in each of the two parentheses in the second to last expression telescope resulting in $\frac{1}{2}(-1)_{N+1} - 0 = (-q)_N$. Thus \eqref{finalequality} is proved, which, in turn, gives
\begin{equation*}
 V(q, N+1) - V(q, N)=\frac{q^{N+1} (-q)_N}{1-q^{N+1}},
\end{equation*}
so that by the principle of mathematical induction, we finally deduce that $U(q, N)=V(q, N)$.
\qed\\

Theorem \ref{chernbeck} now yields a nice relation between the generating functions of $N_{\textup{SC}}(n, N)$ and $d(n, N)$. 
\begin{lemma}\label{GFNSCdnN}
Let $N\in\mathbb{N}$. Then
 \begin{equation*}
  2 (-q)_N \sum_{n=1}^{\infty} N_{\textup{SC}}(n, N)q^n - \sum_{n=1}^{N}\left[\begin{matrix} N \\ n \end{matrix} \right] \frac{q^{n(n+1)/2}}{1-q^{n}}  = \sum_{n=1}^{\infty}d(n, N)q^n. 
 \end{equation*}
\end{lemma}
\begin{proof}
By Theorem \ref{nscN},
\begin{align}\label{NSCanNlink}
2(-q)_{N}\sum_{n=1}^{\infty}N_{\textup{SC}}(n, N)q^n&=2(-q)_{N}\sum_{n=1}^{N}\frac{q^n(q^{n+1})_{N-n}}{(q^{2n};q^2)_{N-n+1}}\nonumber\\
&= 2(-q)_{N}\sum_{n=1}^{N} \frac{q^n}{(1-q^n) (-q^n)_{N-n+1}}\nonumber\\
&=2\sum_{n=1}^{N} \frac{q^n (-q)_{n-1}}{1-q^n}\nonumber\\
&=2\sum_{n=1}^{\infty} a(n, N)q^n\nonumber\\
&=2\sum_{n=1}^{\infty}\left(\textup{ssptd}_{o}(n, N) - \textup{ssptd}_{o}(n-N, N)\right)q^n,
\end{align}
where in the penultimate and ultimate steps we used \eqref{anNgen} and Theorem \ref{chernbeck} respectively.

We now define two functions analogous to $\textup{ssptd}_{o}(n, N)$. Let $\textup{ssptd}(n, N)$ denote 
the sum of smallest parts in all partitions $\pi$ of $n$ into distinct parts and satisfying $l(\pi) - s(\pi)\leq N-1$ and $\textup{ssptd}_e(n, N)$, the same with the added restriction that there be an even number of parts. 

As proved in the beginning of Section \ref{intro}, the right-hand side of van Hamme's identity \eqref{hammeid} generates that of \eqref{guozeng}. Similar to this, it is easy to see that
\begin{align}\label{simvan}
\sum_{n=1}^{N}\left[\begin{matrix} N \\ n \end{matrix} \right] \frac{q^{n(n+1)/2}}{1-q^{n}}=\sum_{n=1}^{\infty}\left(\textup{ssptd}(n, N) - \textup{ssptd}(n-N, N)\right)q^n.
\end{align}
Hence from \eqref{NSCanNlink} and \eqref{simvan} and recalling the definition of $t(n, N)$ from \eqref{tnN},
\begin{align*}
&2 (-q)_N \sum_{n=1}^{\infty} N_{\textup{SC}}(n, N)q^n - \sum_{n=1}^{N}\left[\begin{matrix} N \\ n \end{matrix} \right] \frac{q^{n(n+1)/2}}{1-q^{n}}\nonumber\\
&=\sum_{n=1}^{\infty}\left\{\left(\textup{ssptd}_{o}(n, N) - \textup{ssptd}_{e}(n, N)\right)-\left( \textup{ssptd}_{o}(n-N, N)- \textup{ssptd}_{e}(n-N, N)\right)\right\}q^n\nonumber\\
&=\sum_{n=1}^{\infty}\left(t(n, N)-t(n-N, N)\right)q^{n}\nonumber\\
&=\sum_{n=1}^{\infty}d(n, N)q^n,
\end{align*}
where the last step follows from \eqref{guozeng}. This proves the result.
\end{proof}
We now give an application of Corollary \ref{c=-1finalcor} which gives a new representation for the generating function of $d(n, N)$.
\begin{corollary}\label{dnNGF}
Let $n\in\mathbb{N}$. 
\begin{equation}\label{eqndnN}
 \sum_{n=1}^{N} \frac{(-q)_n}{(q)_n}\frac{q^n}{1-q^n} - 2\sum_{n=1}^{N} \left[\begin{matrix} N \\ n \end{matrix}\right] \frac{q^{n(n+3)/2} (-q)_{n-1}}{(q)_n (1-q^n)} = \sum_{n=1}^{N} \frac{q^n}{1-q^n}.
\end{equation}
\end{corollary}
\begin{proof}
Multiply both sides of Corollary \ref{c=-1finalcor} by $2(-q)_N/(q)_N$ and then simplify the resultant using Lemma \ref{GFNSCdnN} to obtain 
\begin{equation}\label{cor2.13middle}
\sum_{n=1}^{\infty}d(n, N)q^n + \sum_{n=1}^{N} \left[\begin{matrix} N \\ n \end{matrix} \right]\frac{q^{n(n+1)/2}}{1-q^n}\frac{(-q)_n}{(q)_n} = \frac{1}{2} \left\{\frac{(-q)_N}{(q)_N} - 1\right\} + \sum_{n=1}^{N} \frac{(-q)_n}{(q)_n}\frac{q^n}{1-q^n}
\end{equation}
Next, \eqref{qchuvan} with $a=c=1$ gives
\begin{equation*}
\frac{(-q)_N}{(q)_N}-1=\sum_{n=1}^{N}\left[\begin{matrix} N \\ n \end{matrix} \right]\frac{(-1)_n q^{n(n+1)/2}}{(q)_n}.
\end{equation*}
Employing the above equation in \eqref{cor2.13middle} yields \eqref{eqndnN} upon simplification.
\end{proof}

\section{A relation between $d(n, N)$ and a finite analogue of the largest parts function}\label{furcor}
Let $\textup{lpt}(n, N)$ denote the number of occurrences of the largest parts in those partitions $\pi$ of $n$ whose corresponding largest parts are less than or equal to $N$. In what follows, we give a relation between $d(n, N)$ and $\textup{lpt}(n, N)$.
\begin{theorem}\label{finc0}
Let $N\in\mathbb{N}$. Then
\begin{align}
\sum_{n=1}^{N}\frac{q^n}{1-q^n}+ \sum_{n = 1}^{N}\left[\begin{matrix} N \\ n \end{matrix} \right]\frac{q^{n(n+1)}}{(q)_n (1 - q^n)} = \ \sum_{n = 1}^{N}\frac{q^n}{(1 - q^n)(q)_n}. \label{c=0}
	\end{align}
Hence if $\nu(i)$ denotes the number of occurrences of the integer $i$ in a partition $\pi$ of some number $m$ and 
\begin{equation}\label{wnN}
w(m, N):=\sum_{n=1}^{N} \sum_{\substack{\pi \in \mathcal{P}(m, N + 1)\\ \text{exactly $n$ parts} > n}} \left\{\nu(n) + 1\right\},
\end{equation}
then
\begin{equation}\label{idtyy}
d(m, N)+w(m, N)=\textup{lpt}(m, N).
\end{equation}
\end{theorem}
\begin{proof}
Letting $c=0$ in Theorem \ref{fingenc}, and then employing a recent result of Merca \cite[Section 2, Equation (8)]{merca}, namely,
\begin{equation*}
\frac{1}{(q)_N} \sum_{n = 1}^{N} \left[\begin{matrix} N \\ n \end{matrix} \right](-1)^{n - 1}n q^{n(n + 1)/2}=\sum_{n=1}^{N}\frac{q^n}{1-q^n},
\end{equation*}
we obtain \eqref{c=0}. Clearly, the right-hand side of \eqref{c=0} is the generating function of $\textup{lpt}(m, N)$. So we need only interpret
$\displaystyle\sum_{n = 1}^{N}\left[\begin{matrix} N \\ n \end{matrix} \right]\frac{q^{n(n+1)}}{(q)_n (1 - q^n)}$. Now $\left[\begin{matrix} N \\ n \end{matrix}\right]$ is the generating function of the number of partitions into at most $n$ parts each $\leq N-n$.
A typical partition from this set would look like $a_1 + a_2 + \cdots + a_n$ with 
$N-n \geq a_1 \geq a_2 \geq \cdots \geq a_n \geq 0.$ Add $n$ copies of $n + 1$ to this
partition, one copy to each of the $a_i$'s (corresponding to the $q^{n(n+1)}$ term in the numerator of the sum). The partition now becomes $(a_1+n+1)+(a_2+n+1)+ \cdots + (a_n+n+1)$, where each $a_i+n+1$ satisfies $n+1 \leq a_i+n+1 \leq N+1$. So we get a partition, say $\lambda$, into $n$ parts each part being greater than $n$. The term 
$\frac{1}{(q)_n (1-q^n)}$ remains to be dealt with. Since
\begin{align*}
\frac{1}{(q)_n (1-q^n)} = \ \frac{1}{(1-q)}\cdot\frac{1}{(1-q^2)} \cdots \frac{1}{(1-q^{n-1})} \sum_{\nu(n)=0}^{\infty}(\nu(n)+1)q^{\nu(n)n},
\end{align*}
the left-hand side represents partitions into parts $\leq n$ with weight $\nu(n)+1$. Putting this together with $\lambda$, we obtain a partition with each part $\leq N+1$, exactly $n$ parts greater than $n$, and with weight $\nu(n)+1$. Thus the coefficient of $q^m$
in $\displaystyle\sum_{n = 1}^{N}\left[\begin{matrix} N \\ n \end{matrix} \right]\frac{q^{n(n+1)}}{(q)_n (1 - q^n)}$ is 
$\displaystyle\sum_{n=1}^{N} \sum_{\substack{\pi \in \mathcal{P}(m, N+1) \\ \text{exactly $n$ parts} > n}} \left\{\nu(n)+1\right\}=w(m, N)$, by \eqref{wnN}.
Thus we obtain \eqref{idtyy}.
\end{proof}

\section{An example of a generalized sum-of-tails identity and its combinatorial implication}\label{6.15}

The identity in \eqref{atw} was instrumental in the first proof of Lemma \ref{sumexpliciteval} of Section \ref{nscsection}. This identity appears to be new and is an example of a generalization of `sum-of-tails identity', a topic first initiated by Ramanujan \cite[p.~14]{lnb} and later developed by Andrews \cite{andrews1986}, Andrews and Freitas \cite{af}, Andrews, Jim\'{e}nez-Urroz and Ono \cite{ajo}, Patkowski \cite{patkowskicm}, and Zagier \cite{zagiertop}, to name a few. 

Add the $j=0$ term of the sum on the left-hand side of \eqref{atw}, that is, $(q)_{\infty}-1$, to both sides of \eqref{atw}, then replace $j$ by $j-1$ and $m$ by $m+1$ to obtain
\begin{equation}\label{8.1}
\sum_{j=1}^{\infty}(-1)^{j-1}((q^{j})_{m}-1)=\frac{1}{2}\left(\frac{(q)_{m}}{(-q)_{m}}-1\right).
\end{equation}
The partition-theoretic interpretation of the above identity is
\begin{equation}\label{pti}
\sum_{\pi\in\mathcal{D}(n, m)}(-1)^{\#(\pi)}\sum_{j=\xi}^{s(\pi)}(-1)^{j-1}=\frac{1}{2}\sum_{\substack{\pi\in\overline{\mathcal{P}}(n)\\ l(\pi)\leq m}}(-1)^{\#(\pi)},
\end{equation}
where $\xi:=\max\{1, l(\pi)-(m-1)\}$, $\overline{\mathcal{P}}(n)$ is the collection of overpartitions of $n$, and, as mentioned in the introduction, $\mathcal{D}(n, m)$ is the collection of partitions of $n$ into distinct parts such that $l(\pi)-s(\pi) \leq m-1$. Clearly, the right-hand side of \eqref{8.1} is the generating function of that of \eqref{pti}. Now consider the summand on the left-hand side of \eqref{8.1}. Any partition $\pi$ that is counted by $(q^{j})_m-1$ is a partition into distinct parts counted with weight $(-1)^{\#(\pi)}$ and satisfying $s(\pi)\leq j$ and $l(\pi)\leq m+j-1$. This necessitates $l(\pi)-s(\pi)\leq m-1$. Thus $j$, the index of summation on the left-hand side of \eqref{pti} runs from $l(\pi)-m+1$ to $s(\pi)$. However, $l(\pi)-m+1$ might be negative, which is why we have to take the lower limit of summation to be $\xi=\max\{1, l(\pi)-(m-1)\}$. This completes the proof of \eqref{pti}.

The limiting case $m\to \infty$ of \eqref{8.1} is well-known, and is the first equality below:
\begin{equation}\label{conject}
\sum_{j=1}^{\infty}(-1)^{j-1}((q^{j})_{\infty}-1)=\frac{1}{2}\left(\frac{(q)_{\infty}}{(-q)_{\infty}}-1\right)=\sum_{j=1}^{\infty}(-1)^jq^{j^2},
\end{equation}
where the second equality follows from Gauss' identity $(q)_{\infty}/(-q)_{\infty}=1+2\sum_{j=1}^{\infty}(-1)^jq^{j^2}$. The above identity is precisely Equation (3.22) in \cite{agl13}.

Another proof of \eqref{conject} can be obtained by letting $t=q, a=0$, and $g_n=(-1)^n/(q)_n$ (so that $g(x):=\sum_{n=0}^{\infty}g_nx^n=1/(-x;q)_{\infty}$) in Theorem 4.1 of Andrews and Freitas \cite{af}.

Note also that an analogous identity
\begin{equation*}
\sum_{j=0}^{\infty}((q^{j+1})_{\infty}-1)=-\sum_{n=1}^{\infty}\frac{q^n}{1-q^n}.
\end{equation*}
is well-known \cite[p.~14, Equations (12.41), (12.45)]{fine}.

\section{A finite analogue of Garvan's identity and its special cases}\label{fingi}
We prove Theorem \ref{fingithm} in this section. Its corollaries are then discussed. We begin with some definitions and lemmas. Let $N\in\mathbb{N}$. We define the finite analogue of Fine's function \eqref{finefunction} by
\begin{equation}\label{finefunctionfin}
F_{N}(\a,\b;\tau)=F_{N}(\a,\b;\tau:q):=\sum_{n=0}^{N}\left[\begin{matrix} N\\n\end{matrix}\right]\frac{(\alpha q)_{n}(\tau)_{N-n}(q)_n \tau^{n}}{(\beta q)_{n}(\tau)_N }.
\end{equation}
We need the partial fraction decomposition of $F_{N}(\a, \b; t)$ which generalizes \eqref{fine16.3}.
\begin{lemma}\label{pfdfin}
For $N\in\mathbb{N}$, we have
\begin{equation*}
F_{N}(a, b; t)=\frac{(1-t q^{N})(a q)_N}{(b q)_{N}}\sum_{n=0}^{N}\left[\begin{matrix} N\\n\end{matrix}\right]\frac{(b/a)_n(a q)_{N-n}(a q)^n}{(a q)_{N}(1-t q^n)}.
\end{equation*}
\end{lemma}
\begin{proof}
Let $\a=b/a, \b=t, \g=tq$ and $\tau=aq$ in \eqref{fht} and simplify using \eqref{e2}.
\end{proof}
\begin{lemma}[Finite analogue of the Rogers-Fine identity]\label{frfi}
Let $N\in\mathbb{N}$. For $\b\neq 0$,
\begin{align*}
F_{N}(\a,\b;\tau)=(1- \tau q^{N})\sum_{n=0}^{N}\left[\begin{matrix} N\\n\end{matrix}\right]\frac{(\alpha q)_n(q)_n\left(\frac{\alpha\tau q}{\beta}\right)_n(\alpha \tau q^2)_{N-1}(\tau\beta)^nq^{n^2}(1-\alpha\tau q^{2n+1})}{(\beta q)_n(\tau)_{n+1}(\alpha \tau q^2)_{N+n}}.
\end{align*}
\end{lemma}
\begin{proof}
Let $a=\tau\alpha q,$ $b=\frac{\alpha\tau q}{\beta}$, $d=\alpha q$, $e=q$ and $c\rightarrow \infty$ in Watson's $q$-analogue of Whipple's theorem, that is, \eqref{watson87}, and simplify.
\end{proof}
\begin{remark}
Letting $N\to\infty$ in the above lemma gives the well-known Rogers-Fine identity \cite[p.~15, Equation (14.1)]{fine}:
\begin{equation*}
F(\a, \b; \tau)=\sum_{n=0}^{\infty}\frac{(\a q)_n\left(\frac{\a\tau q}{\b}\right)_n(\tau\beta)^nq^{n^2}(1-\a\tau q^{2n+1})}{(\beta q)_{n}(\tau)_{n+1}}.
\end{equation*}
\end{remark}
Let
\begin{align}\label{s1zqn}
S_1(z, q, N):=\sum_{n=1}^{N}\left[\begin{matrix} N\\n\end{matrix}\right]_{q^2}\frac{ (q^2;q^2)_{n}(q^2;q^2)_{n-1}(zq;q^2)_{N-n}(zq)^{n}}{(zq^2;q^2)_{n}(zq;q^2)_{N}}.
\end{align}
Using \eqref{finefunctionfin}, it can be easily expressed in terms of the finite analogue of Fine's function:
\begin{equation}\label{sf}
S_1(z, q, N)=\frac{zq(1-q^{2N})}{(1-zq^2)(1-zq^{2N-1})}F_{N-1}(1,zq^2,zq:q^2).
\end{equation}
In the following two lemmas, we derive two representations for $S_1(z, q, N)$, which will be crucial in proving Theorem \ref{fingithm}.
\begin{lemma}\label{interlemma}
Let $N\in\mathbb{N}$. Then
\begin{align}\label{inter}
S_1(z, q, N)=\sum_{n=1}^{N}\left[\begin{matrix} N\\n\end{matrix}\right]_{q^2} \frac{(q;q^2)_{n-1}(q^2;q^2)_{n}(zq^2;q^2)_{N-n} z^nq^{2n-1}}{(zq;q^2)_{n}(zq^2;q^2)_{N}}.
\end{align}
\end{lemma}
\begin{proof}
Note that
\begin{align}\label{ee0}
S_1(z, q, N)&=\frac{1}{(1-zq^{2N-1})}\sum_{n=0}^{N-1}\left[\begin{matrix} N\\n+1\end{matrix}\right]_{q^2}\frac{ (q^2;q^2)_{n+1}(q^2;q^2)_n(zq;q^2)_{N-n-1}(zq)^{n+1}}{(zq^2;q^2)_{n+1}(zq;q^2)_{N-1}}\nonumber\\
&=\frac{zq(1-q^{2N})}{(1-zq^{2N-1})(1-zq^2)}\sum_{n=0}^{N-1}\frac{ (q^2;q^2)_{N-1}(zq;q^2)_{N-n-1}(q^2;q^2)_n}{(q^2;q^2)_{N-n-1}(zq;q^2)_{N-1}(zq^4;q^2)_n}(zq)^{n}\nonumber\\
&=\frac{zq(1-q^{2N})}{(1-zq^{2N-1})(1-zq^2)}\sum_{n=0}^{N-1}\frac{ (q^{2-2N};q^2)_n(q^2;q^2)_{n}}{(q^{3-2N}/z;q^2)_n(zq^4;q^2)_{n}}q^{2n}\nonumber\\
&=\frac{zq(1-q^{2N})}{(1-zq^{2N-1})(1-zq^2)}{}_3 \phi _2
\left[\begin{matrix}
q^{2-2N}, &q^2, &q^2 \\
zq^4, & \frac{q^{3-2N}}{z}
\end{matrix}; q^2, q^2\right]
\end{align}
where in the penultimate step we used \cite[p.~351, Appendix (I.11)]{gasperrahman}
\begin{equation}\label{ee1}
\frac{(b;q^2)_N}{(a;q^2)_N}\frac{(a;q^2)_{N-n}}{(b;q^2)_{N-n}}\left(\frac{a}{b}\right)^n=\frac{(q^{2-2N}/b;q^2)_n}{(q^{2-2N}/a;q^2)_n}
\end{equation}
with $a=zq, b=q^2$ and $N$ replaced by $N-1$. Now use \eqref{corfht} with $N$ and $q$ respectively replaced by $N-1$ and $q^2$, and then let $\alpha=\beta=q^2$, $\tau=zq$, and $\gamma=zq^4$ to transform the ${}_3\phi_{2}$ in \eqref{ee0} so as to obtain
\begin{align*}
S_1(z, q, N)&=\frac{zq(1-q^{2N})}{(1-zq)(1-zq^{2N})}{}_3 \phi _2
\left[\begin{matrix}
q^{2-2N}, &q, &q^2 \\
zq^3, & \frac{q^{2-2N}}{z}
\end{matrix}; q^2, q^2\right]\nonumber\\
&=\frac{zq(1-q^{2N})}{(1-zq)(1-zq^{2N})}\sum_{n=0}^{N-1}\frac{(q^2;q^2)_{N-1}(zq^2;q^2)_{N-n-1}(q;q^2)_nz^nq^{2n}}{(zq^2;q^2)_{N-1}(q^2;q^2)_{N-n-1}(zq^3;q^2)_n}\nonumber\\
&=\sum_{n=1}^{N}\left[\begin{matrix} N\\n\end{matrix}\right]_{q^2} \frac{(q;q^2)_{n-1}(q^2;q^2)_{n}(zq^2;q^2)_{N-n} z^nq^{2n-1}}{(zq;q^2)_{n}(zq^2;q^2)_{N}},
\end{align*}
where in the second step we again used \eqref{ee1} with $N$ replaced by $N-1$, $a=zq^2$ and $b=q^2$. This proves \eqref{inter}.
\end{proof}
On page $5$ of Ramanujan's Lost Notebook \cite{lnb} (see also \cite[p.~29, Entry 1.7.2]{abramlostII} we find the following identity valid for $|b|<1$ and $a\in\mathbb{C}$:
\begin{align}\label{ramfina}
\sum_{n=0}^{\infty} \frac{(-1)^n (-q;q)_n (- \frac{aq}{b} ; q)_n \, b^n }{(aq; q^2)_{n+1} }  = \sum_{n=0}^{\infty}  \frac{(-1)^n (- \frac{aq}{b} ; q)_n \, b^n q^{\frac{n(n+1)}{2}} }{(-b; q)_{n+1}}.
\end{align}
In the following lemma, we derive a finite analogue of the special case of \eqref{ramfina} when $a=-b=zq$, which is needed in proving Theorem \ref{fingithm}.
\begin{lemma}\label{1.7.2finthm}
Let $N\in\mathbb{N}$. Let $S_1(z, q, N)$ be defined in \eqref{s1zqn}. Then
\begin{align}\label{1.7.2fin}
S_1(z, q, N)=\sum_{n=1}^{N}\left[\begin{matrix} N\\n\end{matrix}\right]_{q^2}\left(\frac{(q;q)_{2n-2}z^{2n-1}q^{n(2n-1)}}{(zq;q)_{2n-1}}+\frac{(q;q)_{2n-1}z^{2n}q^{n(2n+1)}}{(zq;q)_{2n}}\right)\frac{(q^{2};q^2)_n}{(zq^{2N+1};q^2)_{n}}.
\end{align}
\end{lemma}
\begin{proof}
Replace $q$ by $q^2$ in Lemma \ref{frfi}, let $\a=-\frac{a}{bq}, \b=-b$ and $\tau=-bq$, and then multiply both sides of the resulting identity by $\frac{(1-aq^{2N+2})}{(1+b)(1+bq^{2N+1})}$ to deduce that
\begin{align*}
&\left(1-aq^{2N+2}\right)\sum_{n=0}^{N}\frac{(-1)^n\left(-\frac{aq}{b};q^2\right)_n(-bq;q^2)_{N-n}(q^2;q^2)_{N}b^nq^n}{(q^2;q^2)_{N-n}(-b;q^2)_{n+1}(-bq;q^2)_{N+1}}\nonumber\\
&=\sum_{n=0}^{N}\frac{(q^2;q^2)_N(\frac{-aq}{b};q)_{2n}(aq^4;q^2)_Nb^{2n}q^{2n^2+n}(1-aq^{4n+2})}{(q^2;q^2)_{N-n}(-b;q)_{2n+1}(aq^4;q^2)_{N+n}(1+bq^{2n+1})}\nonumber\\
&=\sum_{n=0}^{N}\frac{(q^2;q^2)_N(\frac{-aq}{b};q)_{2n}(aq^4;q^2)_Nb^{2n}q^{2n^2+n}}{(q^2;q^2)_{N-n}(-b;q)_{2n+1}(aq^4;q^2)_{N+n}}\left(1-\frac{bq^{2n+1}+aq^{4n+2}}{1+bq^{2n+1}}\right)\nonumber\\
&=\sum_{n=0}^{N}\left(\frac{(\frac{-aq}{b};q)_{2n}b^{2n}q^{n(2n+1)}}{(-b;q)_{2n+1}}-\frac{(\frac{-aq}{b};q)_{2n+1}b^{2n+1}q^{(2n+1)(n+1)}}{(-b;q)_{2n+2}}\right)\frac{(q^2;q^2)_N(aq^4;q^2)_N}{(q^2;q^2)_{N-n}(aq^4;q^2)_{N+n}}.
\end{align*}
Now let $a=zq=-b$ to get
\begin{align*}
&(1-zq^{2N+3})\sum_{n=0}^{N} \frac{(q;q^2)_{n}(q^2;q^2)_N(zq^2;q^2)_{N-n}z^n q^{2n}} {(zq;q^2)_{n+1}(q^2;q^2)_{N-n}(zq^2;q^2)_{N+1}}\nonumber\\
&=\sum_{n=0}^{N}\left(\frac{(q;q)_{2n}z^{2n}q^{2n^2+3n}}{(zq;q)_{2n+1}}+\frac{(q;q)_{2n+1}z^{2n+1}q^{(2n+1)(n+2)}}{(zq;q)_{2n+2}}\right)\frac{(q^2;q^2)_N(zq^5;q^2)_N}{(q^2;q^2)_{N-n}(zq^5;q^2)_{N+n}}.
\end{align*}
Now replace $n$ by $n-1$ on both sides, then multiply both sides of the resulting equation by $\frac{zq(1-q^{2N+2})}{(1-zq^{2N+3})}$, and then replace $N$ by $N-1$ in the resulting identity to obtain after simplification
 \begin{align}\label{j2}
&\sum_{n=1}^{N}\left[\begin{matrix} N\\n\end{matrix}\right]_{q^2} \frac{(q;q^2)_{n-1}(q^2;q^2)_{n}(zq^2;q^2)_{N-n} z^nq^{2n-1}}{(zq;q^2)_{n}(zq^2;q^2)_{N}}\nonumber\\
&=\sum_{n=1}^{N}\left[\begin{matrix} N\\n\end{matrix}\right]_{q^2}\left(\frac{(q;q)_{2n-2}z^{2n-1}q^{n(2n-1)}}{(zq;q)_{2n-1}}+\frac{(q;q)_{2n-1}z^{2n}q^{n(2n+1)}}{(zq;q)_{2n}}\right)\frac{(q^{2};q^2)_n}{(zq^{2N+1};q^2)_{n}}.
\end{align}
Now \eqref{1.7.2fin} follows from \eqref{inter} and \eqref{j2}.
\end{proof}
We have now collected all tools necessary to prove Theorem \ref{fingithm}.
\begin{proof}[Theorem \textup{\ref{fingithm}}][]
Replace $q$ by $q^2$, $z$ by  $z/q$, then $c$ by $z$ in Theorem \ref{fingGaravan} to get
\begin{align}\label{ee-3}
 \sum_{n=1}^{N}\left[\begin{matrix} N\\n\end{matrix}\right]_{q^2}\frac{(-1)^{n-1}z^nq^{n^2}(q^2;q^2)_n}{(zq;q^2)_n(1-zq^{2n})}=\sum_{n=1}^{N}\left[\begin{matrix} N\\n\end{matrix}\right]_{q^2}\frac{(q;q^2)_{n-1}(q^2;q^2)_{n} (zq^2;q^2)_{N-n} z^{n}q^{2n-1}}{(zq;q^2)_{n}(zq^2;q^2)_N}.
\end{align}
Now Theorem \ref{fingithm} follows from \eqref{ee-3}, \eqref{s1zqn} and Lemmas \ref{interlemma} and \ref{1.7.2finthm}.
\begin{remark}
From \eqref{sf}, and Lemma \textup{\ref{pfdfin}} with $q$ and $N$ replaced by $q^2$ and $N-1$ respectively and then with $a=1, b=zq^2$ and $t=zq$, we see that
\begin{align*}
S_1(z, q, N)=\frac{zq(q^2;q^2)_N}{(zq^2;q^2)_N}\sum_{n=0}^{N-1}\frac{(zq^2;q^2)_nq^{2n}}{(q^2;q^2)_n(1-zq^{2n+1})},
\end{align*}
which is a simple representation for $S_1(z, q, N)$.
\end{remark}

\end{proof}
Letting $z=1$ in Theorem \ref{fingithm}, Lemmas \ref{interlemma} and \ref{1.7.2finthm}, simplifying, and then combining the resulting identities together, we get
\begin{corollary}Let $N\in\mathbb{N}$. We have
\begin{align}
 &\sum_{n=1}^{N}\left[\begin{matrix} N\\n\end{matrix}\right]_{q^2}\frac{(-1)^{n-1}q^{n^2}(q^2;q^2)_{n-1}}{(q;q^2)_n}=\sum_{n=1}^{N}\frac{q^{2n-1}}{1-q^{2n-1}}
=\sum_{n=1}^{N} \left[\begin{matrix} N\\n\end{matrix}\right]_{q^2}\frac{(q^2;q^2)_{n-1}(q;q^2)_{N-n}q^{n}}{(q;q^2)_{N} }\nonumber\\
&=\sum_{n=1}^{N}\left[\begin{matrix} N\\n\end{matrix}\right]_{q^2}\frac{(q^2;q^2)_{n-1}(1-q^{4n-1})q^{n(2n-1)}}{(q^{2N+1};q^2)_{n}(1-q^{2n-1})},\label{ga3sz1}
\end{align}
\end{corollary}
The limiting case $N\to\infty$ of the above identity is well-known (see \cite[Remark 3]{dixitmaji18}). 

We now give partition-theoretic interpretation of the first equality in the above corollary. This generalizes Corollary 1.3 (i) of Garvan \cite{garvan1}.

Consider $S_3=\left\{\vec{\pi} =(\vec{\pi_1},\vec{\pi_2})\in V_3: |\vec{\pi}|=|\vec{\pi_1}|+|\vec{\pi_2}|\right\}$ and $\vec{V_3}=D_{E,n,N}\times P^*_{O,N}$ where $D_{E,n,N}$ denotes the set of partitions of a number into distinct even parts lying between $[2N-2n+2,2N]$ and $P^*_{O,N}$ is the set of partitions of a number in which all parts except possibly the largest part are odd and all odd positive integers less than or equal to the largest part occur as parts.
\begin{corollary}
Let $\nu_d(\l)$ be as defined in the introduction, and let $S_3$ and $V_3$ be defined as above. Let $\ell_{O}(\l)$ denote the largest odd part of the partition $\l$. Define $w(\vec{\pi}):=(-1)^{\nu_{d}(\pi_1)+\frac{1}{2}(\ell_{O}(\pi_2)-1)}$. Let $d_1(m, N)$ denote the number of odd divisors of $m$ which are less than or equal to $2N-1$. Then
\begin{equation*}
\sum_{n=1}^{N}\sum_{\substack{\vec{\pi}\in S_3\\|\vec{\pi}|=m}}w(\vec{\pi})=d_1(m, N).
\end{equation*}
\end{corollary}
\begin{proof}
We first show
\begin{equation}\label{lh1}
\sum_{n=1}^{N}\frac{(q^2;q^2)_N}{(q^2;q^2)_{N-n}}\frac{(-1)^{n-1}q^{n^2}}{(1-q^{2n})(q;q^2)_n}=\sum_{m=1}^{\infty}\left( \sum_{n=1}^{N}\sum_{\substack{\vec{\pi}\in S_3\\|\vec{\pi}|=m}}w(\vec{\pi})\right)q^m.
\end{equation}
To that end, note that $\frac{(q^2;q^2)_N}{(q^2;q^2)_{N-n}}=(q^{2N-2n+2};q^2)_n$ implies that it generates partitions $\pi_1$ coming from the set $D_{E,n,N}$ with weight $(-1)^{\nu_d(\pi_1)}$. Also,
\begin{equation*}
\frac{(-1)^{n-1}q^{n^2}}{(1-q^{2n})(q;q^2)_n}=(-1)^{n-1}\frac{q}{1-q}\cdot\frac{q^3}{1-q^3}\cdots\frac{q^{2n-1}}{1-q^{2n-1}}\cdot\frac{1}{1-q^{2n}}
\end{equation*}
implies that it generates the partitions $\pi_2$ coming from $P^*_{O,N}$ with weight equal to $(-1)^{\frac{1}{2}(\ell_{O}(\pi_2)-1)}$. Combining the two establishes \eqref{lh1}. Finally,
\begin{align*}
\sum_{n=1}^{N}\frac{q^{2n-1}}{1-q^{2n-1}}=\sum_{n=1}^{N}\sum_{m=1}^{\infty}q^{m(2n-1)}= \sum_{k=1}^{\infty}\left(\sum_{\substack{d|k, d\hspace{1mm}\textup{odd}\\d\leq 2N-1}}1\right)q^k
\end{align*}
implies that it generates $d_1(m, N)$. Thus, the above interpretations of the expressions in the first equality of \eqref{ga3sz1} establish the corollary.
\end{proof}
Another corollary of Theorem \ref{fingithm} is now presented.
\begin{corollary}
Let $N\in\mathbb{N}$. Then
\begin{align}\label{corfingiagl}
&\frac{1}{(1-q)}\left(1-\frac{(q^2;q^2)_{N}}{(q^3;q^2)_{N}}\right)=\sum_{n=1}^{N}\left[\begin{matrix} N\\n\end{matrix}\right]_{q^2}\frac{(-1)^{n-1}q^{n(n+1)}}{1-q^{2n+1}}=\sum_{n=1}^{N}\left[\begin{matrix} N\\n\end{matrix}\right]_{q^2}\frac{(q;q^2)_{n-1}(q^3;q^2)_{N-n}q^{3n-1}}{(q^3;q^2)_N}\nonumber\\
&=\sum_{n=1}^{N}\frac{(q^2;q^2)_{n-1}}{(q^3;q^2)_n}q^{2n}=(1-q)\sum_{n=1}^{N}\left[\begin{matrix} N\\n\end{matrix}\right]_{q^2}\frac{(q^2;q^2)_{n-1}(1-q^{4n})q^{2n^2+n-1}}{(q^{2N+2};q^2)_{n}(1-q^{2n-1})(1-q^{2n+1})}.
\end{align}
\end{corollary}
\begin{proof}
Let $z=q$ in Theorem \ref{fingithm}. Also, replace $q$ by $q^2$ in Corollary \ref{fin_AGL} and then let $c=q$. Combining the two resulting identities lead to \eqref{corfingiagl}.
\end{proof}
\section{Concluding Remarks}\label{cr}
In this paper, we have obtained, among other things, the finite analogues of rank and crank for vector partitions. It would be worthwhile seeing if rank and crank exist for ordinary partitions enumerated by $p(n, N)$. 

We give below another representation for the finite analogue of the crank generating function $\frac{(q)_N}{(zq)_{N}(z^{-1}q)_{N}}$ which portrays the possibility of mimicking the approach that Andrews and Garvan \cite[Equation (2.3)]{andrewsgarvan88} used to obtain crank of an ordinary partition enumerated by $p(n)$. Note that letting $\a=-zq$ and $\tau=z^{-1}$ in \eqref{rowyee} gives
\begin{equation*}
\frac{(q^2)_{N-1}}{(z^{-1}q)_{N}}=\frac{1}{(1-q^{N+1})}\sum_{k=0}^{N}\left[\begin{matrix} N \\ k \end{matrix} \right]\frac{(zq)_k(z^{-1}q)^k}{(z^{-1}q^{N-k+1})_{k}}.
\end{equation*}
Hence,
\begin{align}
&\frac{(q)_N}{(zq)_{N}(z^{-1}q)_{N}}\nonumber\\
&=\frac{1-q}{(1-q^{N+1})(zq)_{N}}+\frac{(1-q)}{(1-q^{N+1})}\sum_{k=1}^{N}\left[\begin{matrix} N \\ k \end{matrix} \right]\frac{q^kz^{-k}}{(zq^{k+1})_{N-k}(z^{-1}q^{N-k+1})_k}\label{cnzq1}\\
&=\frac{1-q}{(zq)_{N}}+(1-q)\sum_{k=1}^{N}\left[\begin{matrix} N \\ k \end{matrix} \right]\frac{q^kz^{-k}}{(zq^{k+1})_{N-k}(z^{-1}q^{N-k+1})_k}\nonumber\\
&\quad+\frac{q^{N+1}}{(1-q^{N+1})}\left(\frac{1-q}{(zq)_{N}}+(1-q)\sum_{k=1}^{N}\left[\begin{matrix} N \\ k \end{matrix} \right]\frac{q^kz^{-k}}{(zq^{k+1})_{N-k}(z^{-1}q^{N-k+1})_k}\right).\label{cnzq2}
\end{align}
Note that letting $N\to\infty$ in either \eqref{cnzq1} and \eqref{cnzq2} gives
\begin{equation*}
\frac{1-q}{(zq)_{\infty}}+\sum_{k=1}^{\infty}\frac{q^kz^{-k}}{(q^2)_{k-1}(zq^{k+1})_{\infty}},
\end{equation*}
which was combinatorially interpreted by Andrews and Garvan \cite[Equation (2.3)]{andrewsgarvan88} thereby obtaining the crank for an ordinary partition. Unfortunately, we are unable to proceed beyond \eqref{cnzq1} or \eqref{cnzq2}.

Do there exist congruences for $p(n, N)$ which could be combinatorially explained by our finite analogues of rank and crank for vector partitions? It would also be worthwhile to see if there exists a refinement for $\textup{spt}(n, N)$ of Andrews' famous congruences for $\textup{spt}(n)$, namely \cite{andrews08},
\begin{align*}
\textup{spt}(5n+4)&\equiv0\pmod{5},\nonumber\\
\textup{spt}(7n+5)&\equiv0\pmod{7},\nonumber\\
\textup{spt}(13n+6)&\equiv0\pmod{13}.
\end{align*}
Moreover, it would be interesting to see if a congruence of the form \eqref{kronreh} holds for $\textup{spt}(n, N)$.

In \cite{garvan10}, Garvan conjectured that for any even $k \geq 2$, the inequality 
\begin{equation*}
M_{k}(n) > N_{k}(n)
\end{equation*}
must be true for all $n \geq 1$.  
For a sufficiently large $n$ and fixed $k$, this was proved by Bringmann and Mahlburg \cite{BM09}, and by Bringmann, Mahlburg and Rhoades \cite{BMR11} by analyzing the asymptotic behavior of the difference $M_{k}(n) - N_{k}(n)$.  Later Garvan \cite{garvan11} himself proved his conjecture for all $n$ and $k$ by finding a combinatorial interpretation for the difference between symmetrized crank and rank moments. 
In what follows, we give a similar conjecture for the difference of finite crank and rank moments.  
\begin{conjecture}
For any fixed natural number $N$ and even $k>2$, 
\begin{equation*}
M_{k,N}(n) > N_{k,N}(n) \quad \mathrm{for \,\, all} \,\, n\geq 1. 
\end{equation*}
\end{conjecture}
For $k=2$, the result is already shown to be true in Corollary \ref{inq_fin_crank_rank}. We have numerically verified, with the help of \textit{Mathematica}, that the above conjecture holds at least for $4\leq k \leq 12$ and $1\leq n \leq 20$. 

Finally, considering the enormous impact and applications of the theory of Fine's function $F(\a, \b; t)$ developed in \cite{fine}, it would be worthwhile to do the same for its finite analogue $F_{N}(\a, \b; t)$ defined in \eqref{finefunctionfin}. In this paper, we have only obtained two results for $F_{N}(\a, \b; t)$, namely, its partial fraction decomposition and the finite version of the Rogers-Fine identity, because obtaining those was essential to proving Theorem \ref{fingithm}. We also note that another finite analogue of Fine's function was studied in \cite{andrewsbell}.

\end{document}